\documentclass[12pt, a4paper]{amsart}
\usepackage[a4paper, left=2cm, right=2cm, top=2cm, bottom=2cm]{geometry}
\usepackage{amssymb}
\usepackage{amsmath}
\usepackage{stmaryrd}
\usepackage{xcolor}
\makeatother
\usepackage{hyperref}
\def\a{\alpha}

\def\r{\rho}

\def\G{{\Gamma}}

\def\C{{\mathbb C}}
\def\R{{\mathbb R}}
\def\Z{{\mathbb Z}}
\def\Q{{\mathbb Q}}

\def\Ker{\mathop{\rm Ker}\nolimits}

\def\Coker{\operatorname{Coker}}
\def\Id{\operatorname{Id}}
\def\tr{\mathop{\rm tr}\nolimits}
\def\diag{\mathop{\rm diag}\nolimits}

\def\Tr{\operatorname{Tr}}

\def\ln{\operatorname{ln}}

\def\ord{\operatorname{ord}}
\def\T{{\mathbb T}}
\def\1{\mathbf 1}

 \def\rk{\operatorname{rk}}

\def\Coin{\operatorname{Coin}}

\newcommand{\ov}[1]{\overline{#1}}

\newcommand{\wh}[1]{\widehat{#1}}

\newtheorem{theorem}{Theorem}
\newtheorem{lemma}[theorem]{Lemma}
\newtheorem{example}[theorem]{Example}
\newtheorem{remark}[theorem]{Remark}

\DeclareMathOperator{\fix}{Fix}
\def\Fix{\operatorname{Fix}}

\def\N{{\mathbb N}}

\begin{document}

\title[The Reidemeister and the Nielsen numbers]
{The Reidemeister and the Nielsen numbers: growth rate, asymptotic behavior, dynamical zeta functions  and  the Gauss congruences }

\author{Alexander Fel'shtyn and Mateusz Slomiany}

\address{Alexander Fel'shtyn, Instytut Matematyki, Uniwersytet Szczecinski,
ul. Wielkopolska 15, 70-451 Szczecin, Poland}
\email{alexander.felshtyn@usz.edu.pl}

\address{Mateusz Slomiany, Instytut Matematyki, Uniwersytet Szczecinski,
ul. Wielkopolska 15, 70-451 Szczecin, Poland}
\email{mateusz.slomiany@phd.usz.edu.pl}


\begin{abstract}
In the present paper, taking a dynamical point on view,  we    study the growth rate and   asymptotic behavior of the sequences  of  the Reidemeister numbers and the sequences  of  the Reidemeister and the Nielsen coincidence numbers. We also prove the Gauss congruences for  the sequence   $\{R(\varphi^n,\psi^n)\}$ of  the
Reidemeister coincidence numbers of the tame pair $(\varphi,\psi)$ of endomorphisms of a
torsion-free nilpotent group~$G$ of finite Pr\"ufer rank.
Furthermore, we prove the rationality of the Nielsen coincidence zeta function,
the Gauss congruences for  the sequence $\{N(f^n, g^n)\}$ of the Nielsen coincidence numbers
and show that the growth rate exists for the sequence  \{$N(f^n, g^n)\}$
of  tame pair of maps $(f,g)$ of a compact nilmanifold to itself.
\end{abstract}

\maketitle
\setcounter{tocdepth}{1}
\tableofcontents

\section{Introduction}\label{sec:intro}
In topological fixed point theory,  the Reidemeister and  the Nielsen numbers
arise as homotopy invariants associated to  a continuous
self-map of a connected compact polyhedron.  Passing to the
fundamental group of the polyhedron, the Reidemeister numbers admit an
algebraic treatment in terms of twisted conjugacy classes; compare
\cite{FelshB}.
We assume $X$ to be a connected, compact polyhedron and $f:X\rightarrow X$ to be a continuous map. Let $p:\tilde{X}\rightarrow X$ be the universal cover of $X$ and $\tilde{f}:\tilde{X}\rightarrow \tilde{X}$ a lifting
of $f$, {i.e.,} $p\circ\tilde{f}=f\circ p$.
 Two liftings $\tilde{f^\prime}$ and $\tilde{f}$  of $f$ are said to be \emph{conjugate} if there
 exists covering translation $ \gamma \in \Gamma\cong\pi_1(X)$, such that $\tilde{f^\prime}= \gamma \circ \tilde{f} \circ  \gamma^{-1}.$
Lifting classes are equivalence classes by conjugacy. Notation:
$
[\tilde{f}]= \{ \gamma \circ \tilde{f} \circ  \gamma^{-1}  |\gamma \in \Gamma \}.
$
The subset $p(\fix(\tilde{f}))\subset \fix(f)$ is called the \emph{fixed point class} of $f$ determined by the lifting class $[\tilde{f}]$.
 Our definition of a fixed point class is via the universal covering space. It essentially
 says: two fixed point of $f$ are in the same class iff there is a lifting $\tilde f $ of $f$ having
 fixed points above both of them. There is another way of saying this, which does not use covering space explicitly, hence is  very useful in identifying fixed point classes.
 Two fixed points $x_0$ and $x_1$ of $f$ belong to the same fixed point class iff
 there is a path $c$ from $x_0$ to $x_1$ such that $c \cong f\circ c $ (homotopy relative
    endpoints - ``Nielsen disk").
 This fact  can be considered as an equivalent definition (geometrical) of a non-empty fixed point class. Every map $f$  has only finitely many non-empty fixed point classes, each a compact
 subset of $X$.
 A fixed point class is called \emph{essential} if its index is nonzero. The number of lifting classes of $f$ (and hence the number of all fixed point classes) is called the \emph{Reidemeister number} of $f$, denoted {by} $R(f)$. This is a positive integer or infinity. The number of essential fixed point classes is called the \emph{Nielsen number} of $f$, denoted by $N(f)$.
The Nielsen number is always finite. $R(f)$ and $N(f)$ are homotopy invariants. In the category of compact, connected polyhedra the Nielsen number of a map is, apart from in certain exceptional cases, equal to the least number of fixed points of maps with the same homotopy type as $f$.

Let $f, g: X \to Y$ be two maps between two compact connected manifolds.
Let
$\Coin(f,g) := \{ x \in X \mid f(x) = g(x) \}$ be the set of the coincidence  points of $f$ and $g$.
Two points of coincidence $x_1,x_2 \in \Coin(f,g)$ are said to be \emph{Nielsen equivalent} if there is a path
$\alpha:[0,1] \to X$ from $x_1$ to $x_2$ such that $f \circ \alpha$ is homotopic to $g \circ \alpha$.
If $[C]$ denotes the fixed-endpoint homotopy class of $C$
then we can denote it by $[f \circ \alpha] = [g \circ \alpha]$.
Being Nielsen equivalent is obviously an equivalence relation.
We denote the set of the coincidence classes by $\overline{\Coin}(f,g)$.
Let $\{f_t\}$ be a homotopy of $f$ and let $\{g_t\}$ be a homotopy of $g$, $f_0 = f, g_0 = g$.
Two points of coincidence $x_1 \in \Coin(f,g)$ and $x_2 \in \Coin(f_1, g_1)$ are $\{f_t\},\{g_t\}$-related
if there exists a~path $C:[0,1] \to X$ with $C(0) = x_1, C(1) = x_2$ such that
$[\{f_t \circ C(t)\}] = [\{g_t \circ C(t)\}]$.
Two equivalence classes $\alpha \in \overline{\Coin}(f,g)$ and $\beta \in \overline{\Coin}(f_1,g_1)$
are $\{f_t\},\{g_t\}$-related if every point of coincidence in $\alpha$ is $\{f_t\},\{g_t\}$-related
to every point of coincidence in $\beta$.

A coincidence class $\alpha \in \overline{\Coin}(f,g)$ is \emph{essential}
if for any homotopies $\{f_t\}$ of $f$ and $\{g_t\}$ of $g$, there is a coincidence class
in $\overline{\Coin}(f_1,g_1)$ to which $\alpha$ is $\{f_t\},\{g_t\}$-related.
The finite number of essential coincidence classes in $\overline{\Coin}(f,g)$ is called
the \emph{Nielsen coincidence number} of $f,g$ and denoted by $N(f,g)$,
see also \cite{Wong}.

Let $G$ be a  group and $\varphi: G\rightarrow G$ an
endomorphism. Two elements $x,y\in G$ are said to be
$\varphi$-{\em conjugate} or {\em twisted conjugate,}
if and only if there exists $g \in G$ such that $ y=g  x   \varphi(g^{-1}).$
We will write $\{x\}_\varphi$ for the $\varphi$-{\em conjugacy} or
{\em twisted conjugacy} class
of the element $x\in G$.
The number of $\varphi$-conjugacy classes is called the {\em Reidemeister number}
of an  endomorphism $\varphi$ and is  denoted by $R(\varphi)$.
If $\varphi$ is the identity map then the $\varphi$-conjugacy classes are the usual
conjugacy classes in the group $G$.

We have the Reidemeister bijection between  fixed point classes, lifting classes and corresponding twisted conjugacy
classes. Therefore the
Reidemeister number of continuous map  $f$ coincides with Reidemeister number
(the number of twisted conjugacy classes) of induced endomorphism $f_*$ of the fundamental group of $X$: $R(f)=R(f_*)$.
In this paper we take mostly a group-theoretic point of view,
inspired by, but otherwise  independent of the topological
origins of the subject.

Let $G$ be a  group and $\varphi, \psi : G\rightarrow G$
two endomorphisms. Two elements $\alpha,\beta\in G$ are said to be $(\varphi, \psi)$\textit{-conjugate} if and only if there exists $g \in G$ with
$
\beta=\psi(g)  \alpha   \varphi(g^{-1}).
$
The number of $(\varphi,\psi$)-conjugacy classes is called the Reidemeister coincidence
number of   endomorphisms $\varphi $ and $\psi $, denoted by $ R(\varphi,\psi)$. If $\psi$ is the identity map then the $(\varphi,id)$-conjugacy classes are the $\varphi$-conjugacy classes in the group $G$ and $R(\varphi,id) = R(\varphi)$.
The Reidemeister coincidence number $R(\varphi,\psi)$ has useful applications in Nielsen - Reidemeister  coincidence theory.
We call an  endomorphism $\varphi$  \emph{tame} if the
Reidemeister numbers $R(\varphi^n)$ are finite for all $n \in
\mathbb{N}$.
We call the pair $(\varphi,\psi)$ of endomorphisms \emph{tame} if the coincidence
Reidemeister numbers $R(\varphi^n,\psi^n)$ are finite for all $n \in
\mathbb{N}$.

Maps $f,g$ induce homomorphisms $f_\#, g_\#$ between fundamental groups of compact manifolds $X$ and $Y$.
We define the \emph{topological} Reidemeister coincidence number of maps $R(f,g) := R(f_\#,g_\#)$
to be the Reidemeister coincidence number of the~induced homomorphisms.
If $Y=X$ then we are going to say that the pair of maps $f,g$ is \emph{tame} if
 $R(f^n, g^n) < \infty$ for every $n \in \N$.

In the theory of dynamical systems, the Reidemeister coincidence
numbers count the synchronisation points of two
maps, i.e. the points whose orbits intersect under simultaneous
iterations of two maps; see \cite{Mi13}, for instance.

In the present paper, taking a dynamical point on view, we study the growth rate
and asymptotic behavior of the sequences of the Reidemeister, and the Nielsen numbers
$\{R(\varphi^n)\}$, $\{R(\varphi^n,\psi^n)\}$, $\{R(f^n)\}$, $\{N(f^n, g^n)\}$
for a tame endomorphism $\varphi$, and a tame map $f$, and
for a tame pair $(\varphi,\psi)$ of endomorphisms, and a tame pair  $(f,g)$ of maps.
We also  prove the Gauss congruences for  the sequence $\{R(\varphi^n,\psi^n)\}$ of the
Reidemeister coincidence numbers of the tame pair $(\varphi,\psi)$ of endomorphisms of a
torsion-free nilpotent group~$G$ of finite Pr\"ufer rank.
Furthermore, we prove the rationality of the Nielsen coincidence zeta function and
the Gauss congruences for the sequence $\{N(f^n, g^n)\}$ of the Nielsen coincidence numbers
and show that the growth rate exists for the sequence $\{N(f^n, g^n)\}$
for a tame pair of maps $(f,g)$ of a compact nilmanifold to itself.

This work is organized as follows.

In section \ref{sec:preliminary}, we remind definitions, facts and notation from other theories
 which are of  importance in course of the paper.
They regard topological entropy of a map, unitary dual of a group,
isolated lower central series of a~nilpotent group, and $p$-adic numbers.

In section \ref{sec:growth-reidemeister}, we show that for a tame endomorphism of a
torsion-free nilpotent group~$G$ of finite Pr\"ufer rank the growth rate
of the sequence $\{R(\varphi^n)\}$ of the Reidemeister numbers
exists and we provide a closed formula
for the growth rate via the eigenvalues of induced endomorphisms on abelian sections.

In section \ref{sec:growth-entropy}, we establish a connection between the growth rate
of the sequence $\{R(\varphi^n)\}$ of the Reidemeister numbers and topological entropy
of unitary duals maps induced on the unitary duals of abelian sections by a tame endomorphism
of a finitely generated torsion-free nilpotent group~$G$.

In section \ref{sec:coincidence-reidemeister},  we show that for a tame pair
$(\varphi,\psi)$  of endomorphisms of a
torsion-free nilpotent group~$G$ of finite Pr\"ufer rank the growth rate of the sequence
$\{R(\varphi^n,\psi^n)$\} of the Reidemeister coincidence numbers exists
and we provide a closed formula for the growth rate via the eigenvalues
of induced endomorphisms on abelian sections.

In section \ref{sec:congruences}, we prove the Gauss congruences for the sequence
$\{R(\varphi^n,\psi^n)\}$ of the
Reidemeister coincidence numbers of the tame pair $(\varphi,\psi)$ of endomorphisms of a
torsion-free nilpotent group~$G$ of finite Pr\"ufer rank.

In section \ref{sec:coincidence-nielsen}, we prove the rationality of the Nielsen coincidence zeta function
and
the Gauss congruences for the sequence $\{N(f^n, g^n)\}$ of the Nielsen coincidence numbers
and show that the growth rate of the sequence \{$N(f^n, g^n)\}$ of the
Nielsen coincidence numbers exists for a tame pair of maps  $(f,g)$
of a compact nilmanifold to itself.

In section \ref{sec:asymptotic}, we study asymptotic behavior of the sequences
$\{R(\varphi^n,\psi^n)\}$,  $\{N(f^n, g^n)\}$
of the Reidemeister coincidence numbers  of the tame pair $(\varphi,\psi)$ of endomorphisms of a
torsion-free nilpotent group~$G$ of finite Pr\"ufer rank and the Nielsen coincidence  numbers
 for a tame pair of maps  $(f,g)$
of a compact nilmanifold to itself.

\section{Preliminary considerations }\label{sec:preliminary}

\subsection{Topological entropy and  growth rate  of the Nielsen numbers of iterations}\label{EC}

The most widely used measure for the complexity of a dynamical system is the topological
entropy. We remind its definition.
Let $ f: X \rightarrow X $ be a self-map of a compact metric space. For a given $\epsilon > 0 $
and $ n \in \N$, a subset $E \subset X$ is said to be \emph{$(n,\epsilon)$-separated} under $f$ if for
each pair $x \not= y$ in $E$ there is $0 \leq i <n $ such that $ d(f^i(x), f^i(y)) \geq \epsilon$.
Let $s_n(\epsilon,f)$  denote the largest cardinality of any $(n,\epsilon)$-separated subset $E$
under $f$. Thus  $s_n(\epsilon,f)$ is the greatest number of orbit segments ${x,f(x),\cdots,f^{n-1}(x)}$
of length $n$ that can be distinguished one from another provided we can only distinguish
between points of $X$ that are  at least $\epsilon$ apart.
Now let
$$
  h(f,\epsilon):= \limsup_{n} \frac{1}{n}\log \,s_n(\epsilon,f)
$$
$$
  h(f):=\lim_{\epsilon \rightarrow 0} h(f,\epsilon).
$$
The number $0\leq h(f) \leq \infty $, which is independent of the metric $d$ used,
providing the induced topology is the same,
is called the topological entropy of $f$.
If $ h(f,\epsilon)> 0$ then, up to resolution $ \epsilon >0$, the number $s_n(\epsilon,f)$ of
distinguishable orbit segments of length $n$ grows exponentially with $n$. So $h(f)$
measures the growth rate in $n$ of the number of orbit segments of length $n$
with arbitrarily fine resolution.

The following example will be used to show a connection between the growth rate of Reidemeister numbers and topological entropy in Section \ref{sec:growth-entropy}.

\begin{example}[\protect{(see \cite[Example 3.2c]{KatHas})}]\label{ex:entropy-computation}
  We compute topological entropy for the expanding map $E_m: S^1 \to S^1$
  of the unit circle to itself which sends $z \mapsto z^m$ where $m \in \mathbb{Z} \setminus \{-1,0,1\}$.
  We will measure distance between two points $x,y \in S^1$ as a normalized difference of their arguments
  (as complex numbers, let $\iota$ denote the imaginary unit) modulo $1$.
  We are going to compute cardinality of the maximal $(n,\epsilon)$-separeted set.

  Pick $\epsilon = 1/m^{k}$ for some $k \in \mathbb{N}$.
  Then points
  \[
    x_j := \exp \left( \frac{2\pi\iota j}{m^{n+k-1}} \right), \quad j = 0, 1, \dots, m^{n+k-1}-1
  \]
  form an $(n,\epsilon)$-separated set.
  For example for two points next to each other we have
  \[
    d \left(E_m^{n-1} (x_j), E_m^{n-1} (x_{j+1}) \right)
    = \frac{(j+1)m^{n-1} - jm^{n-1}}{m^{n+k-1}}
    = \frac{m^{n-1}}{m^{n+k-1}} = \frac{1}{m^k} = \epsilon.
  \]
  It also means that we cannot add further points to the set so that it remains $(n,\epsilon)$-separated
  therefore $s_n(1/m^k, E_m) = m^{n+k-1}$.
  Now directly from the definition
  \begin{align*}
    h(E_m) = \lim_{k \to \infty} \limsup_{n \to \infty} \frac{1}{n} \log m^{n+k-1} = \log |m|.
  \end{align*}
\end{example}

For a ``hyperbolic" (Axiom A) diffeomorphism of a manifold topological entropy
$h(f)$ measures the growth rate of the number of periodic points, see \cite{KatHas}.

A basic relation between  topological entropy $h(f)$ and Nielsen numbers
was found by N. Ivanov.
\begin{theorem}\cite{I} \label{Iv}
  Let $f$ be a continuous map on a compact connected polyhedron $X$. Then
  $$
    h(f) \geq \limsup_{n} \frac{1}{n}\cdot\log N(f^n) =: {\log N^\infty(f)}
  $$
\end{theorem}
The number $N^\infty(f)$ is called the \emph{growth rate} of the Nielsen numbers of iterations of $f$.
This inequality is remarkable in that it does not require smoothness of the map and provides a common lower bound for the topological entropy of all maps in a homotopy class.

\subsection{Unitary dual}
By $\wh G$ we denote the \emph{unitary dual} of $G$, i.e.
the space of equivalence classes of
unitary irreducible representations of $G$, equipped with the
\emph{compact-open} topology.
If $\varphi: G\to G$ is an automorphism,
it induces a dual map $\wh\varphi:\wh G\to\wh G$, $\wh\varphi (\r):=\r\circ\varphi$.
This dual map $\wh\varphi$ define a dynamical system  on the unitary dual space $\widehat{G}$.

The unitary dual of any finitely generated abelian group $G$ is an abelian group $\wh G$ which is,
by Pontryagin duality, the direct sum
of a~Torus whose dimension is the rank of $G$, and of a~finite abelian group.
In that case the dual map will be called the \emph{Pontryagin dual map}.

\subsection{Isolated lower central series of a torsion-free nilpotent group}

Let $G$ be a group. For a~subgroup $H \leqslant G$, we define the \emph{isolator} $\sqrt[G]{H}$ of $H$ in $G$ as
$$
  \sqrt[G]{H} := \{g \in G \mbox{ } | \mbox{ } g^n \in H \mbox{ for some } n \in \mathbb{N} \}.
$$
Note that the isolator of a subgroup $H \leqslant G$ doesn't have to be a subgroup in general.

Now let $G$ be a torsion-free nilpotent group with the lower central series
$$
  G = \gamma_1(G) \geq \gamma_2(G) \geq \dots \geq \gamma_c(G) \geq \gamma_{c+1}(G) = \{e\}
$$
where as usual $\gamma_{k+1}(G) := \llbracket G, \gamma_k(G) \rrbracket$ is the $k$-fold commutator group.
It turns out that, by results of \cite[Lem.~1.1.2 and Lem.~1.1.4]{Dek}, we benefit by considering another central series, namely
\begin{equation}\label{equ:isolated-lcs}
  G = \sqrt[G]{\gamma_1(G)} \geq \sqrt[G]{\gamma_2(G)} \geq \dots \geq \sqrt[G]{\gamma_c(G)} \geq \sqrt[G]{\gamma_{c+1}(G)} = \{e\}
\end{equation}
which we call \emph{isolated} (or \emph{adapted}) \emph{lower central series}.
The isolated lower central series terminates iff $G$ is a torsion-free nilpotent group.
The main advantage of using this series over the usual lower central series is that all the factors
$$
  \sqrt[G]{\gamma_k(G)} / \sqrt[G]{\gamma_{k+1}(G)} \cong \mathbb{Z}^{d_k}, \quad d_k \in \mathbb{N}
$$
are torsion-free.

\subsection{$p$-adics}

In this section we remind basic definitions and prove facts about $p$-adic numbers needed later.
Let $p$ be any prime number. For $a \in \Z \setminus \{0\}$ we define the $p$-adic ordinal $\ord_p a$
to be the highest power of $p$ which divides $a$.
For a rational $q = a/b$ we define $\ord_p q := \ord_p a - \ord_p b$.
The $p$-adic norm of $q$ is defined
\[
  \left| q \right|_p := \begin{cases} \frac{1}{p^{\ord_p q}}, \quad &q \neq 0,\\ 0, \quad &q=0 \end{cases}.
\]
The norm is \emph{non-Archimedean} which means that for all $x, y \in \Q$
it satisfies $\left| x+y \right|_p \leq \max\left( |x|_p, |y|_p \right)$.

We will make use of the following property:
\begin{equation}\label{eq:padic-triangle}
  \big( |x|_p < |y|_p \big) \quad \Longrightarrow \quad |x-y|_p = |y|_p.
\end{equation}
The field $\Q$ with the norm $|\cdot|_p$ is not complete but if we consider the set $\Q_p$
of equivalence classes of Cauchy sequences with regard to the relation
\[
  \{a_i\} \sim \{b_i\} \quad \Longleftrightarrow \quad |a_i - b_i|_p \to 0 \ (i \to \infty)
\]
(just as we actually do when we go from $\Q$ to $\R$ with the usual norm)
with the $p$-adic norm of a sequence defined as $\left|\{a_i\}\right|_p := \lim |a_i|_p$
then $\Q_p$ turns out to be a complete field.
Moreover, from (\ref{eq:padic-triangle}) and definition of Cauchy sequences we get
that in fact the sequence $\{|a_i|_p\}$ of $p$-adic norms of elements of a Cauchy sequence $\{a_i\}$
is constant for sufficiently large $i$.
The following remark will be important later.
\begin{remark}\label{thm:possible-values-of-padic-norm}
  Possible values for $|\cdot|_p$ on $\Q_p$ are in the set $\{p^n\}_{n \in \Z} \cup \{0\}$.
  In particular the greatest possible value $<1$ is $1/p$.
\end{remark}

We are going to work with numbers algebraic over $\Q_p$.
Let $\Q_p(\xi)$ be a finite extension of $\Q_p$ generated by $\xi$ which is a root of a monic polynomial
\[
  x^n + a_1 x^{n-1} + \dots + a_{n-1} x + a_n, \quad a_i \in \Q_p.
\]
We define the \emph{norm} of the extension as $\N_{\Q_p(\xi)/\Q_p}(\xi) := (-1)^n a_n$.
The extension of $p$-adic norm from $\Q_p$ to $\Q_p(\xi)$ is defined as
\[
  |\xi|_p := \left| \N_{\Q_p(\xi)/\Q_p}(\xi) \right|_p^{1/n}.
\]
There exist the algebraic closure $\overline{\Q_p}$ and its completion $\Omega$ (which is algebraically
closed as~well) with further extensions of the norm but we will not need them in this work.
The subset $\Z_p := \{ x \in \Q_p \mid |x|_p \leq 1 \}$ is a ring
and its elements are called \emph{$p$-adic integers}

There is the $p$-adic analog of logarithm defined for all $z \in \Omega$ satisfying $|z|_p < 1$:
\begin{equation*}
  \log_p(1+z) := \sum_{m=1}^{\infty} \frac{(-1)^{m+1}}{m} z^m.
\end{equation*}
We are going to make use of the $p$-adic logarithm to prove a very specific inequality.

\begin{lemma}\label{thm:padic-log-geq-1/n}
  Let $\xi$ be an algebraic number over $\Q_p$ such that $|\xi|_p = 1$ but $\xi$ is not a root of unity.
  Then there is a constant $C$ such that for every $n \in \N$ the inequality
  \begin{equation}\label{eq:1/n-leq-x-1}
    \frac{1}{n} \leq C \cdot \left| \xi^n - 1 \right|_p
  \end{equation}
  holds. The constant does not depend on $n$.
\end{lemma}
\begin{proof}
  For $|1|_p = 1 = |\xi|_p$, from (\ref{eq:padic-triangle}) we have $|\xi^n-1|_p \leq 1$ for all $n \in \N$.
  If for some $n$ there is equality $|\xi^n-1|_p = 1$ then (\ref{eq:1/n-leq-x-1}) holds for $C=1$

  Now consider $n$ for which $|\xi^n-1|_p < 1$.
  Denote the degree $\left[ \Q_p(\xi) : \Q_p \right] =: k$.
  From remark \ref{thm:possible-values-of-padic-norm} and
  the~definition of the $p$-adic norm for algebraic numbers we have
  \begin{equation}\label{eq:bound-for-padic-norm-of-x-1}
      |\xi^n-1|_p \leq \frac{1}{\sqrt[k]{p}}.
  \end{equation}
  Therefore we can compute the $p$-adic logarithm and
  \begin{align*}
    |n|_p \cdot |\log_p (\xi)|_p &= |\log_p (\xi^n)|_p \\
    &= \left|(\xi^n - 1) - \frac{(\xi^n - 1)^2}{2} + \frac{(\xi^n - 1)^3}{3} - \dots \right|_p \\
    &\leq \max_m \left( \left|\frac{(\xi^n - 1)^m}{m} \right|_p \right) \\
  \end{align*}
  Note that using (\ref{eq:bound-for-padic-norm-of-x-1}) we get
  \begin{equation}\label{eq:padic-maximum-term}
    0 \leq \left| \frac{(\xi^n - 1)^m}{m} \right|_p
    = p^{\ord_p m} \cdot \left| \xi^n - 1 \right|_p^m
    \leq m \cdot \left| \xi^n - 1 \right|_p^m
    \leq m \cdot \left( \frac{1}{\sqrt[k]{p}} \right)^m.
  \end{equation}
  For a constant $q \in (0, 1/\sqrt[k]{p}]$ the derivative
  $\frac{\partial}{\partial x} xq^x = q^x(1+x\ln q)$ vanishes only for
  $x = -1/\ln q$, is positive for smaller $x$ and negative for greater ones.
  Note that $-1/\ln q$ is an increasing function therefore in our situation $-1/\ln q \ \leq \ k/\ln p$.
  It is possible then that for a finite number of initial $m$s the terms over which we take the maximum
  actually grow, but eventually there is a maximum $M = M(n)$ and the sequence goes to $0$ as $m \to \infty$.

  In other words there exists $m_0 = m_0(n)$ satisfying $1 \leq m_0 \leq k/\ln p$, for which
  \begin{equation*} 
    \max_m \left( \left|\frac{(\xi^n - 1)^m}{m} \right|_p \right)
    = |\xi^n-1|_p \cdot \frac{\left| \xi^n-1 \right|_p^{m_0-1}}{|m_0|_p}
  \end{equation*}
  We can bound from above the right side similarly as above:
  \begin{align*}
    |\xi^n-1|_p \cdot \frac{\left| \xi^n-1 \right|_p^{m_0-1}}{|m_0|_p}
    &=  |\xi^n-1|_p \cdot \Big( p^{\ord_p m_0} \cdot \left| \xi^n-1 \right|_p^{m_0-1} \Big) \\
    &\leq |\xi^n-1|_p \cdot \Big( m_0 \cdot \left| \xi^n-1 \right|_p^{m_0-1} \Big)
  \end{align*}
  For $x \in [1, k/\ln p]$ the derivative $\frac{\partial}{\partial q} xq^{x-1} = x(x-1)q^{x-2} \geq 0$ so
  \begin{align*}
    |\xi^n-1|_p \cdot \Big( m_0 \cdot \left| \xi^n-1 \right|_p^{m_0-1} \Big)
    \ \leq \  |\xi^n-1|_p \cdot \left( m_0 \cdot \frac{1}{(\sqrt[k]{p})^{m_0-1}} \right)
  \end{align*}
  Note that
  \[
    M^* := \max_{1 \ \leq \ m \ \leq \ k/\ln p} \left\{ m \cdot p^{-(m-1)/k} \right\}
    \ \geq \ m_0 \cdot p^{-(m_0-1)/k}
  \]
  for all possible $m_0$, therefore $M^*$ does not depend on $n$. We got
  \[
    \frac{1}{n} \ \leq \ |n|_p \ \leq \ |\xi^n-1|_p \cdot \frac{M^*}{|\log_p(\xi)|_p}
  \]
  and (\ref{eq:1/n-leq-x-1}) holds for $C = M^* / |\log_p(\xi)|_p$.

  In general (\ref{eq:1/n-leq-x-1}) holds for all $n \in \N$ regardless of actual values of $|\xi^n-1|_p$
  if we put $C := \max \left\{1, M^* / |\log_p(\xi)|_p \right\}$.
\end{proof}

\bigskip
\section{Growth rate of the sequence $\{R(\varphi^n)\}$ of the Reidemeister numbers}\label{sec:growth-reidemeister}

Let $\varphi: G \to G$ be a tame endomorphism of a group $G$.
We introduce the growth rate  $R^{\infty}(\varphi)$ and  upper and lower growth rates
$R^{\infty}_{+}(\varphi)$  and  $R^{\infty}_{-}(\varphi)$ of the Reidemeister numbers
of iterations as follows
\begin{gather*}
  R^\infty_{+}(\varphi):=\limsup_{k\to\infty}R(\varphi^k)^{1/k}, \qquad
  R^\infty_{-}(\varphi):=\liminf_{k\to\infty}R(\varphi^k)^{1/k},\\
    R^\infty(\varphi):=\lim_{k\to\infty}R(\varphi^k)^{1/k}.
\end{gather*}
In general there is no reason to expect the upper (lower) limit above to be a limit,
therefore the growth rate $R^\infty(\varphi)$ may not exist.

\begin{example}\label{ex:compute-growth}
  Let $G = \Z$ and $\varphi(x) = dx, d \in \Z \setminus  \{-1, 0, 1\}$.
  Two elements $x, y \in G$ are in the same $\varphi^n$-conjugacy class if there exists an $a \in G$ such that
  \begin{equation*}
      y = a + x - d^na = x + a(1-d^n)
  \end{equation*}
  which is the same as to say that $x$ and $y$ are in the same coset of $(1-d^n)\Z$.
  It is clear that $R(\varphi^n) = |\Z / (1-d^n)\Z| = |d^n - 1|$.
  Now we can compute the growth rate
  \begin{equation*}
      R^{\infty}(\varphi) = \lim_{k\to\infty} \sqrt[k]{|d^k - 1|} = |d|.
  \end{equation*}
\end{example}

Let $\varphi \colon G \to G$ be an endomorphism of a group $G$.
Then $\varphi$  induces a continuous endomorphism
$\overline{\varphi} \colon \overline{G} \to \overline{G}$
of the profinite completion $\overline{G}$ of $G$, and a
natural map on the set of the Reidemeister classes
$$
\mathcal{R}(\varphi) \to
\mathcal{R}(\overline{\varphi}), \quad
[x]_{\varphi} \mapsto [ \iota x
]_{\overline{\varphi}},
$$
where $\iota: G \to  \overline{G}$ is the completion map.

\begin{lemma}[\protect{\cite[Sec.~5.3.2]{FelshB}, \cite{FelsKlop}}]
  \label{lem:prof-compl}
  In the situation described above, the following properties hold.
  \begin{enumerate}
  \item[(i)] If $R(\varphi) < \infty$, then the natural map
    $\mathcal{R}(\varphi) \to
    \mathcal{R}(\overline{\varphi})$ is surjective.
  \item[(ii)] If $G$ is abelian and $R(\varphi) < \infty$ then
    the natural map
    $\mathcal{R}(\varphi) \to
    \mathcal{R}(\overline{\varphi})$ is bijective.
  \end{enumerate}
\end{lemma}

A group $G$ is of \emph{finite (Pr\"ufer) rank} $\rk(G)$ if each its finitely generated subgroup has a generating set of cardinality $\le \rk(G)$ and
$\rk(G)$ is a minimal such number.

Denote by $\mathbb{P}$
the set of all rational primes; for $p \in \mathbb{P}$, the field of
$p$-adic numbers is denoted by $\mathbb{Q}_p$, the ring of $p$-adic
integers by $\mathbb{Z}_p$, and the $p$-adic absolute value (as well
as its unique extension to  a fixed  algebraic closure
$\overline{\mathbb{Q}}_p$) by $\lvert \cdot \rvert_p$.  The absolute
value on $\mathbb{C}$ is denoted by $\lvert \cdot \rvert_\infty$.

\begin{theorem} \label{thm:main-result-nilpotent}
  Let $\varphi \colon G \to G$ be a tame endomorphism of a
  torsion-free nilpotent group~$G$ of finite Pr\"ufer rank.  Let $c$
  denote the nilpotency class of~$G$ and, for $1 \le k \le c$, let
  $\a_k: A_k \to A_k$ denote the induced
  endomorphisms of the torsion-free abelian factor groups
  $A_k=\sqrt[G]{\gamma_k(G)} / \sqrt[G]{\gamma_{k+1}(G)}$ of finite
  rank, $d_k\geq 1$ say, that arise from the isolated lower central
  series  of~$G$.  Then the following hold.

  \smallskip
  \rm{(a)} For $1 \le k \le c$, let
  $$
    \a_{k,\mathbb{Q}} \colon
    A_{k,\mathbb{Q}} \to
    A_{k,\mathbb{Q}}
  $$
  denote the extensions of $\a_k $ to the divisible hull
  $A_{k,\mathbb{Q}} = \mathbb{Q} \otimes_\mathbb{Z} A_k \cong
  \mathbb{Q}^{d_k}$ of~$A_k$.  Let $\xi_{k,1},\ldots, \xi_{k,d_k}$  be the eigenvalues of
  $\a_{k,\mathbb{Q}}$  in a fixed
  algebraic closure of the field~$\mathbb{Q}$, including
  multiplicities.  Set
  $L_k = \mathbb{Q}(\xi_{k,1}, \ldots, \xi_{k,d_k})$; for each $p \in \mathbb{P}$, fix
  some embeddings $L_k \hookrightarrow \overline{\mathbb{Q}}_p$ and  $L_k \hookrightarrow \mathbb{C}$.

  Then there exist subsets $I_k(p) \subseteq \{1,\ldots,d_k\}$, for
  $p \in \mathbb{P}$, such that the following hold.
  \begin{enumerate}
  \item[(i)] For each $p \in \mathbb{P}$, the polynomial
    $\prod_{i \in I_k(p)} (X - \xi_{k,i})$ has coefficients in
    $\mathbb{Z}_p$; in particular,
    $\lvert \xi_{k,i} \rvert_p  \le 1$ for
    $i \in I_k(p)$.
  \item[(ii)] For each $n \in \mathbb{N}$,
  \begin{align} \label{equ:key-formula}
  \begin{split}
    R(\varphi^n) &= \prod_{k=1}^c\prod_{p \in \mathbb{P}}
    \prod_{i \in I_k(p)} \lvert \xi_{k,i}^{\, n} - 1
    \rvert_p^{\, -1} \\
    &= \prod_{k=1}^c\prod_{i=1}^{d_k } \left(\lvert \xi_{k,i}^{\, n} - 1 \rvert_\infty \cdot \prod_{p \in \mathbb{P}}
    \prod_{i \not \in I_k(p)} \lvert \xi_{k,i}^{\, n} - 1 \rvert_p\right);
  \end{split}
  \end{align}
    as this number is a positive integer,
    $\lvert \xi_{k,i}^{\, n} - 1 \rvert_p = 1$ for
    $1 \le i \le d_k$ for almost all~$p \in \mathbb{P}$.
  \end{enumerate}

  \smallskip

 \rm{(b)} Suppose that, for each $k \in \{1, \ldots, c\}$, the cardinality of the set
$$
\mathbf{P}_k:=\left\{p\in \mathbb{P}\colon \prod_{i \not\in I_k(p)} \lvert \xi_{k,i}^{\, n} - 1 \rvert_p \neq 1\right\}
$$
is finite  and that
  $| \xi_{k,i} |_\infty \neq 1$ for $1 \le i \le d_k$.

  Then the growth rate of the  Reidemeister numbers of iterations exists and is given by
  \begin{align}\label{eq:key-formula-growth}
  \begin{split}
    R^{\infty}(\varphi) &=R^{\infty}_{+}(\varphi)=R^{\infty}_{-}(\varphi)= \\
    &= \prod_{k=1}^{c}
    \prod_{i=1}^{d_k} \max\{|\xi_{k,i}|_\infty,1\}\cdot\prod_{k=1}^{c}\prod_{p \in \mathbf{P}_k}
    \prod_{i \in S^*_k(p)} \max \{ \lvert \xi_{k,i} \rvert_p, 1\},
  \end{split}
  \end{align}
  where  $S^*_k(p) = \{i \in \{ \{1, \ldots, d_k\} \setminus I_k(p)  \}\mid \lvert \xi_{k,i} \rvert_p \neq 1\}$
\end{theorem}
\begin{proof}
 For a tame endomorphism $\varphi$, Theorem 6 in \cite{polyc} implies that
 $$
   R(\varphi^n)= \prod_{k=1}^c R(\a_k^{\, n}), \qquad \text{for $n \in \mathbb{N}$.}
 $$
 To prove the assertion (a) and the formula  (\ref{equ:key-formula})
 we follow closely the proofs of Proposition 3.4 and Theorem 1.4 in \cite{FelsKlop}.
 Using Lemma \ref{lem:prof-compl}(2), we
 may pass to the profinite completion of $\alpha_k$.
 For $1 \le k \le c$, the profinite completion $\overline{A_k}$ of the abelian group $A_k$ and its
 endomorphism $\overline{\a_k}$ decompose as direct products
 \[
 \iota \colon A_k \to \overline{A_k} = \prod_{p \in \mathbb{P}}
 (\overline{A_k})_p \qquad \text{and} \qquad \overline{\a_k} = \prod_{p
   \in \mathbb{P}} (\overline{\a_k})_p,
 \]
 where, for each prime $p$, the Sylow pro-$p$ subgroup of
 $\overline{A_k}$ is the pro-$p$ completion $(\overline{A_k})_p$
 of $A_k$, equipped with the endomorphism
 $(\overline{\a_k})_p \colon (\overline{A_k})_p \to
 (\overline{A_k})_p$. Lemma~\ref{lem:prof-compl}(2) shows that
 $$
 R(\a_k^n) = \prod_{p \in \mathbb{P}}
 R((\overline{\a_k})_p^{\, n}), \qquad
 \text{for $n \in \mathbb{N}$;}
 $$
 in particular, $R(\a_k^n) < \infty$ implies that
 $R((\overline{\a_k})_p^{\, n}) = 1$, for
 almost all $p \in \mathbb{P}$.  Hence the product is only formally
 infinite.

 Fix a prime $p \in \mathbb{P}$.
 On the one hand the pro-$p$ group $(\overline{A_k})_p$ is
 torsion-free, abelian and of rank at most~$d_k$, hence
 $(\overline{A_k})_p \cong \mathbb{Z}_p^{\, d_k(p)}$, where
 $d_k(p) = \mathrm{rk}((\overline{A_k})_p) \le d_k$.
 On the other hand we can arrive at $(\overline{A_k})_p$ in a way which says us more
 about eigenvalues of $(\overline{\a_k})_p$ and their relation to eigenvalues of $\alpha_{k, \Q}$.
 We observe that the group $A_k$ has the same pro-$p$ completion as $A_{k,\Z_p} := \Z_p \otimes_{\Z} A_k$
 (see \cite{FelsKlop} and \cite{RibZal}).
 The $\Z_p$-module $A_{k,\Z_p}$ decomposes as a direct sum of its maximal divisible submodule
 $D \cong \Q_p^{d-d_k(p)}$ and a free $Z_p$-module $A \cong \Z_p^{d_k(p)}$ for $d_k(p) \leq d_k$
 (see \cite[Theorem 20]{Kap}).

 The endomorphism $\alpha_{k,\Q}$ induces an endomorphism $\alpha_{k,\Z_p}: A_{k,\Z_p} \to A_{k,\Z_p}$
 The divisible submodule $D$ is invariant under $\alpha_{k,\Z_p}$ (see \cite[Chapter 5]{Kap})
 thus we obtain an endomorphism $\overline{\alpha}_{k,\Z_p}: A_{k,\Z_p}/D \to A_{k,\Z_p}/D$
 where the quotient group $A_{k,\Z_p}/D \cong A \cong \Z_p^{d_k(p)}$.
 Consequently, there is an isomorphism $\sigma: (\overline{A_k})_p \to A_{k,\Z_p}/D$ such that
 $\overline{\alpha}_{k,\Z_p} \circ \sigma = \sigma \circ (\overline{\a_k})_p$.

 We conclude that
 there exists a subset $I_k(p) \subseteq \{1,\ldots,d_k\}$
 such that the endomorphism $(\overline{\a_k})_p$ has eigenvalues $\xi_{k,i}$, $i \in I_k(p)$
 (i.e. a part of eigenvalues of $\a_{k,\mathbb{Q}}$,
 depending on which eigenvalues are acting on the divisible submodule $D$ and which are on the free module $A$).
 In particular, the coefficients of the characteristic polynomial
 $\prod_{i \in I_k(p)} (X - \xi_{k,i})$ of  $(\overline{\a_k})_p$
 belong to  $\mathbb{Z}_p$,  in particular, $\lvert \xi_{k,i} \rvert_p  \le 1$ for
 $i \in I_k(p)$  (see \cite[Prop. 3.4]{FelsKlop} for more details).

 Finally,
 \[
   R((\overline{\a_k})_p^{\, n})
   = \lvert \Coker((\overline{\a_k})_p^{\, n} - 1)  \rvert
   = \lvert \det( (\overline{\a_k})_p^{\, n} - 1) \rvert_p^{-1}
   = \prod_{i \in I_k(p)} \lvert \xi_{k,i}^{\, n}-1 \rvert_p^{\,-1}.
 \]
 Taking the product over all primes $p$ and then over $1 \le k \le c$, we arrive to the first equality
 in~(\ref{equ:key-formula}).  Using the adelic formula
 $\lvert a \rvert_\infty \cdot \prod_{p \in \mathbb{P}} \lvert a \rvert_p =
 1$,
 for $a \in \mathbb{Q} \setminus \{0\}$, we obtain the second
 equality in~(\ref{equ:key-formula}).
 An assumption that endomorphism $\varphi$ is   tame  implies that eigenvalues $\xi_{k,i}$ are not roots of unity.

 Thus it remains to prove the assertion (b).

 For $1 \le k \le c$ and $p \in \mathbf{P}_k$ we write
 $S_k(p) = \{1, \ldots, d_k\} \setminus I_k(p)$ and
 $S^*_k(p) = \{i \in S_k(p) \mid \lvert \xi_{k,i} \rvert_p \neq 1\}$.  We set
 $$
   b = \prod_{k=1}^c \prod_{p \in \mathbf{P}_k} \prod_{i \in
     S^*_k(p)} \max \{ \lvert \xi_{k,i} \rvert_p, 1\}.
 $$
 Then for $i \in S^*_k(p)$, $|\xi_{k,i}^{\, n}|_p=|\xi_{k,i}|_p^{\, n} \ne 1$ and
 $|\xi_{k,i}^{\, n}-1|_p=\max\{|\xi_{k,i}^{\, n}|_p,1\}
  =\max\{|\xi_{k,i}|_p^{n},1\}
  =\max\{\xi_{k,i},1\}^{\, n}$ .
 From this  and from (\ref{equ:key-formula})  we deduce  that, for $n \in \mathbb{N}$,
  $$
    R(\varphi^n) = g(n) \cdot f(n),
  $$
 where
 \begin{align}\label{eq:gif}
 \begin{split}
    g(n) &= \prod_{k=1}^c \prod_{i=1}^{d_k}
    \lvert \xi_{k,i}^{\, n} - 1 \rvert_\infty \cdot  b^n,\\
     f(n) &= \prod_{k=1}^c \prod_{p \in \mathbf{P}_k}
    \prod_{i \in S_k(p) \setminus S^*_k(p)} \big\vert \xi_{k,i}^n - 1 \big\vert_p.
 \end{split}
 \end{align}

 Now we will rearrange the product
  $\prod_{i=1}^{d_k} \lvert \xi_{k,i}^{\, n} - 1
  \rvert_\infty$.  Complex eigenvalues $\xi_{k,i}$ in the spectrum of
  $\a_{k,\mathbb{Q}}$  appear in pairs with their complex conjugate
  $\ov{\xi_{k,i}}$.  For a complex number $\lambda$ one has
$$
|\lambda^n-1| \cdot |\overline\lambda^n-1|=|\lambda^n-1| ^2=(\lambda^n-1)\cdot (\overline\lambda^n-1).
$$
If $\xi_{k,i}$  are real
  eigenvalues of $\a_{k,\mathbb{Q}}$   then
  we have
  $| \xi_{k,i}^{\, n} - 1 |_\infty= \delta_{1,k,i}^{\,
    n} - \delta_{2,k,i}^{\, n}$, where
  $\delta_{1,k,i} = \max\{\lvert \xi_{k,i} \rvert_\infty,1\} $ and
  $\delta_{2,k,i}=\frac{\xi_{k,i}}{\delta_{1,k,i}}$.
Suppose that $\lambda_1, \overline{\lambda_1},\dots,\lambda_s, \overline{\lambda_s}$ are all complex eigenvalues
and $\xi_{k,i(t)}$, $t=1,\dots,T$, are the real ones.
Then the above two observations show that
\begin{equation}\label{eq:form_of_w}
\prod_{i=1}^{d_k} \lvert \xi_{k,i}^{\, n} - 1
  \rvert_\infty=\sum \pm (\mu_1 \nu_1 \cdots \mu_s \nu_s  \delta_{\epsilon(1),k,i(1)} \cdots  \delta_{\epsilon(T),k,i(T)})^n,
\end{equation}
where $\mu_i$ is $\lambda_i$ or $1$, $\nu_i$ is $\overline{\lambda_i}$ or $1$, $\epsilon(i)$ is $1$ or $2$.
Hence, taking the product over $k$ and incorporating $b$ we obtain
  \begin{equation}\label{eq:gn_w}
    g(n) = \sum_{j \in J} c_jw_j^{\, n},
  \end{equation}
  where $J$ is a finite index set, $c_j =\pm 1$ and
  $ w_j \in \mathbb{C} \setminus \{0\}$, $j\in J$.

  If $S_k(p)\setminus S_k^*(p)=\varnothing$ for all $p\in \mathbf{P}_k $ and all $k=1, . . . , c$,
then~$f(n)\equiv 1$.

 Now suppose that~$S_k(p)\setminus S_k^*(p)\neq\varnothing$ for
some $k\in \{1, \dots, c\}$, $p\in \mathbf{P}_k $.
Then \linebreak $\lim_{n\rightarrow\infty}f(n)^{1/n}=1$.
Indeed for the eigenvalues involved in $f(n)$ we can derive the bounds
\begin{equation}\label{eq:fn_bounds}
  \frac{1}{n} \leq C \cdot | \xi_{k,i}^n - 1 |_p \quad \mbox{ and } \quad | \xi_{k,i}^n - 1 |_p \leq 1
\end{equation}
where $C$ is a constant independent of $n$.
Note that $|\xi_{k,i}|_p = 1$ by construction of $S_k(p)\setminus S_k^*(p)$.
The right inequality follows from (\ref{eq:padic-triangle}) whereas
the left one is proved in lemma \ref{thm:padic-log-geq-1/n}.
The squeeze theorem gives us $\lim_{n\rightarrow\infty}f(n)^{1/n}=1$.

Since~$|\xi_{k,i}|_{\infty}\neq 1 $ for  $i=1,\dots,d_k$,  $k=1, \dots, c $,
 the equality (\ref{eq:form_of_w}) and the explicit form of $\delta_{1,k,i}$ and
  $\delta_{2,k,i}$ imply that
there is a dominant term $w_m$ in the expansion (\ref{eq:gn_w}) for which
$|w_m|_{\infty}=\prod_{k=1}^{c} \prod_{i=1}^{d_k} \max\{|\xi_{k,i}|_\infty,1\}\cdot b $, such that
$|w_m|_{\infty}>|w_j|_{\infty}$ for all $j\neq m$.

Then the growth rate of the  Reidemeister numbers of iterations exists and is given by
  $$R^{\infty}(\varphi)=|w_m|_{\infty}=\prod_{k=1}^{c} \prod_{i=1}^{d_k} \max\{|\xi_{k,i}|_\infty,1\}\cdot\prod_{k=1}^{c}\prod_{p \in \mathbf{P}_k} \prod_{i \in
      S^*_k(p)} \max \{ \lvert \xi_{k,i} \rvert_p, 1\}.$$
\end{proof}

\begin{remark}\label{thm:growth-reidemeister-finitely-generated}
  In the proof of the first part of the theorem \ref{thm:main-result-nilpotent} we passed
  to the profinite completions $(\overline{A_k})_p$ of the factor groups
  and stated that because they are torsion-free, abelian and of rank at most $d_k$
  they are isomorphic to $\mathbb{Z}_p^{d_k(p)}$ where $d_k(p) \leq d_k$.

  In case the group $G$ is not only of finite Pr\"ufer rank but is actually finitely-generated,
  we have equality $d_k(p) = d_k$ and the set $I_k(p) = \{1, \dots, d_k\}$.
  It is because if torsion-free abelian factor groups are finitely generated $A_k \cong \Z^{d_k}$
  then the maximal divisible submodule $D$ of $A_{k,\Z_p}$ is trivial.
  It results in disappearing the $p$-adic part of the formula (\ref{equ:key-formula})
  and consequently (\ref{eq:key-formula-growth}) becomes
  \begin{equation}\label{eq:growth-reideiester-finitely-generated}
    R^{\infty}(\varphi) = \prod_{k=1}^{c} \prod_{i=1}^{d_k} \max\{|\xi_{k,i}|_\infty,1\}.
  \end{equation}
  For more details consult \cite[Theorem 1.4]{FelsKlop}.
\end{remark}

\bigskip
\section{Growth rate of the sequence $\{R(\varphi^n)\}$ of the Reidemeister numbers and topological entropy}\label{sec:growth-entropy}

The following example leads us to results on a connection between the growth rate
of the Reidemeister numbers and the topological entropy.
\begin{example}\label{ex:connection-growth-entropy}
  Consider the group $G = \mathbb{Z}$ and the endomorphism $\varphi: x \mapsto dx$
  for $d \in \mathbb{Z} \setminus \{-1, 0, 1\}$ as in Example \ref{ex:compute-growth}.
  We computed the growth rate $R^\infty(\varphi) = |d|$.

  On the Pontryagin dual $\hat{G} \cong U(1)$ our endomorphism $\varphi$ induces the dual endomorphism $\hat{\varphi}(\chi) = \chi^d$.
  In Example \ref{ex:entropy-computation} we computed topological entropy
  $h_{top}(\hat{\varphi}) = \log |d|$.
  We observe that
  \begin{equation*}
    R^{\infty}(\varphi) = \exp(h_{top}(\hat{\varphi})).
  \end{equation*}
\end{example}

A homeomorphism $f$ of a compact metric space $(X,d)$ is \emph{expansive} if there exists a constant $\epsilon > 0$
such that given any two distinct points $x, y \in X$, there exists $ n \in \Z$ such that $d(f^n(x), f^n(y)) >\epsilon$.

We will need also a notion of the Bowen's \emph{specification} \cite{Bow}.
We cite the definition from \cite{Buzzi}.
The formal definition can seem vague at first so it is good to have the following intuition in mind.
We say that a self-map $f: X \to X$ of a compact metric space has the \emph{specification property}
if, given a precision $\epsilon>0$ and a~number of segments of orbits, one is able to find a~periodic orbit
which, with the precision $\epsilon$, follows each one of them and is moving from one segment to another
in a fixed amount of time which depends only on $\epsilon$.

That being said, let $(x_1, l_1),\dots,(x_r, l_r) \in X \times \mathbb{N}$ where $0 \leq r < \infty$ be given.
For a constant $D \in \mathbb{N}$ we will denote the partial length
$$
  L(s) := l_1 + \dots + l_{s-1} + (s-1)D, \quad 1 \leq s \leq r+1.
$$
The map $f$ has the \emph{specification property} if for every $\epsilon>0$ there exist a constant $D<\infty$
and $z \in X$ such that $f^{L(r+1)}(z) = z$ and
$$
  d(f^{L(s)+k}(z), f^k(x_s)) < \epsilon, \quad 0 \leq k < l_s
$$
for all $s = 1,\dots,r$.

\begin{lemma}[\protect{\cite{Bow}}]\label{B}
 For an expansive and having the specification property homeomorphism $f$ of a compact metric space $X$ the growth rate $F^\infty(f)$ of the number of periodic points exists and is given by
$$ F^\infty(f):=\lim_{k\to\infty} \#(\Fix(f^k))^{1/k} = \exp(h(f)), $$
where $h(f)$ denotes the topological entropy.
\end{lemma}

\begin{theorem}\label{thm:reidem-entropy}
  Let $\varphi: G\to G$ be  a tame  automorphism of a finitely generated discrete Abelian group $G$.
  Suppose that the Pontryagin dual map  $\wh\varphi$ is expansive and having the specification property
  homeomorphism of the Pontryagin dual group $\wh G$.
  Then the growth rate of the  Reidemeister numbers of iterations of $\varphi$  exists and is given by
  \[
    R^{\infty}(\varphi)= \exp(h(\wh\varphi)).
  \]
\end{theorem}
\begin{proof}
  The Pontryagin dual of  the cokernel of $(1-\varphi):G\rightarrow G$ is canonically isomorphic to the
  kernel of the dual map $\widehat{(1-\varphi)}:\hat{G}\rightarrow \hat{G}$.
  Since $R(\varphi^n)=\#\Coker (1-\varphi^n)$ is finite, we have
  \[
    \#\Coker (1-\varphi^n) =\#\widehat{\Coker(1-\varphi^n)}= \#\Ker \widehat{(1-\varphi^n)} .
  \]
  The map $\widehat{1-\varphi^n}$ is equal to $\hat{1}-\hat{\varphi^n}$.
  Its kernel is thus the set of fixed points of the
  map $\hat{\varphi^n}:\hat{G}\rightarrow \hat{G}$.
  We therefore have:
  \begin{equation}
    R(\varphi^n) = \#\fix\left(\hat{\varphi^n}:\hat{G}\rightarrow \hat{G}\right)
  \end{equation}
  Then, Lemma \ref{B} implies the result.
\end{proof}

\begin{theorem}
  Let $G$ be a finitely generated torsion-free nilpotent group with the isolated lower central series
  \rm{\eqref{equ:isolated-lcs}}.
  Let $\varphi: G \to G$ be its automorphism such that all induced automorphisms
  \[
    \varphi_k: \sqrt[G]{\gamma_k(G)} / \sqrt[G]{\gamma_{k+1}(G)}
    \to \sqrt[G]{\gamma_k(G)} / \sqrt[G]{\gamma_{k+1}(G)}
  \]
  fulfill assumptions of \rm{Theorem \ref{thm:reidem-entropy}}. Then
  \[
    R^{\infty}(\varphi) = \exp \left( \sum_{k=1}^c h(\wh{\varphi_k}) \right).
  \]
\end{theorem}
\begin{proof}
  For endomorphisms of finitely generated nilpotent group we have the Roman'kov formula \cite{Rom}
  (for generalisations see also \cite[Theorem 6]{polyc} and \cite[Proposition 2.1]{FelsKlop})
  \[
    R(\varphi) = \prod_{k=1}^{c} R(\varphi_k).
  \]
  For each iteration $n$ of $\varphi_k^n$ the equality (\ref{equ:key-formula}),
  and therefore the whole Theorem \ref{thm:main-result-nilpotent}, holds.
  Hence, we can take the n-th root on both sides of the Roman'kov formula
  and go to the infinity with $n$.
  We get
  \[
    R^\infty(\varphi) = \prod_{k=1}^{c} R^\infty(\varphi_k)
  \]
  and the thesis follows from theorem \ref{thm:reidem-entropy}.
\end{proof}

\bigskip
\section{ Growth rate of the sequence $\{R(\varphi^n,\psi^n)\}$ of the  Reidemeister coincidence numbers}\label{sec:coincidence-reidemeister}

By analogy to the growth rates of the Reidemeister numbers of iterations of one endomorphism
we define the growth rate, upper growth rate and lower growth rate of the Reidemeister coincidence numbers of iterations
\begin{align*}
  R^\infty(\varphi, \psi) &:= \lim_{k\to\infty} R(\varphi^k, \psi^k)^{1/k}\\
  R_+^\infty(\varphi, \psi) &:= \limsup_{k\to\infty} R(\varphi^k, \psi^k)^{1/k}\\
  R_-^\infty(\varphi, \psi) &:= \liminf_{k\to\infty} R(\varphi^k, \psi^k)^{1/k}
\end{align*}
respectively.

The following Theorem is essentially based on \cite[Thm.~1.4]{FelsKlop}

\begin{theorem}\label{thm:growth-coinc}
  Let $\varphi, \psi \colon G \to G$ be a tame pair of endomorphisms of a
  torsion-free nilpotent group~$G$ of finite Pr\"ufer rank.  Let $c$
  denote the nilpotency class of~$G$ and, for $1 \le k \le c$, let
  $\varphi_k ,\psi_k \colon G_k \to G_k$ denote the induced
  endomorphisms of the torsion-free abelian factor groups
  $G_k=\overline{\gamma}_k(G) / \overline{\gamma}_{k+1}(G)$ of finite
  rank, $d_k\geq 1$ say, that arise from the isolated lower central
  series~\eqref{equ:isolated-lcs} of~$G$.  Then the following hold.

  \smallskip

  \rm{(a)} For each $n \in \mathbb{N}$, there is a bijection
  between the set $\mathcal{R}(\varphi^n, \psi^n)$ of
  $(\varphi,\psi)$-Reidemeister coincidence classes and the cartesian
  product $\prod_{k=1}^c \mathcal{R}(\varphi_k^{\, n}, \psi_k^{\, n})$;
  consequently,
  \begin{equation}\label{eq:reidemeister-coincidence-romankov}
    R(\varphi^n, \psi^n) = \prod_{k=1}^c R(\varphi_k^{\, n},
    \psi_k^{\, n}) \qquad \text{for $n \in \mathbb{N}$.}
  \end{equation}

  \rm{(b)} For $1 \le k \le c$, let
  \[
    \varphi_{k,\mathbb{Q}}, \psi_{k,\mathbb{Q}} \colon
    G_{k,\mathbb{Q}} \to
    G_{k,\mathbb{Q}}
  \]
  denote the extensions of $\varphi_k, \psi_k$ to the divisible hull
  $G_{k,\mathbb{Q}} = \mathbb{Q} \otimes_\mathbb{Z} G_k \cong
  \mathbb{Q}^{d_k}$ of~$G_k$.  Suppose that each pair of endomorphisms
  $\varphi_{k,\mathbb{Q}}, \psi_{k,\mathbb{Q}}$ is simultaneously
  triangularisable.  Let $\xi_{k,1}, \ldots, \xi_{k,d_k}$ and
  $\eta_{k,1}, \ldots, \eta_{k,d_k}$ be the eigenvalues of
  $\varphi_{k,\mathbb{Q}}$ and $\psi_{k,\mathbb{Q}}$ in a fixed
  algebraic closure of the field~$\mathbb{Q}$, including
  multiplicities, ordered so that, for $n \in \mathbb{N}$, the
  eigenvalues of
  $\varphi_{k,\mathbb{Q}}^{\, n} - \psi_{k,\mathbb{Q}}^{\, n}$ are
  $\xi_{k,1}^{\, n} - \eta_{k,1}^{\, n}, \ldots, \xi_{k,d_k}^{\, n} -
  \eta_{k,d_k}^{\,n}$.  Set
  $L_k = \mathbb{Q}(\xi_{k,1}, \ldots, \xi_{k, d_k}, \eta_{k,1},
  \ldots, \eta_{k, d_k})$; for each $p \in \mathbb{P}$ fix an
  embedding $L_k \hookrightarrow \overline{\mathbb{Q}}_p$ and choose
  an embedding $L_k \hookrightarrow \mathbb{C}$.

  Then there exist subsets $I_k(p) \subseteq \{1,\ldots,d_k\}$, for
  $p \in \mathbb{P}$, such that the following hold.
  \begin{enumerate}
  \item[(i)] For each $p \in \mathbb{P}$, the polynomials
    $\prod_{i \in I_k(p)} (X - \xi_{k,i})$ and
    $\prod_{i \in I_k(p)} (X - \eta_{k,i})$ have coefficients in
    $\mathbb{Z}_p$; in particular,
    $\lvert \xi_{k,i} \rvert_p , \lvert \eta_{k,i} \rvert_p \le 1$ for
    $i \in I_k(p)$.
  \item[(ii)] For each $n \in \mathbb{N}$,
  \begin{align} \label{equ:key-formula-2}
  \begin{split}
    R(\varphi_k^{\, n},\psi_k^{\, n}) &= \prod_{p \in \mathbb{P}}
    \prod_{i \in I_k(p)} \lvert \xi_{k,i}^{\, n} - \eta_{k,i}^{\, n}
    \rvert_p^{\, -1} =\\
    &= \prod_{i=1}^{d_k }\lvert \xi_{k,i}^{\, n} -
    \eta_{k,i}^{\, n} \rvert_\infty \cdot \prod_{p \in \mathbb{P}}
    \prod_{i \not \in I_k(p)} \lvert \xi_{k,i}^{\, n} - \eta_{k,i}^{\,
      n} \rvert_p;
    \end{split}
  \end{align}
    as this number is a positive integer,
    $\lvert \xi_{k,i}^{\, n} - \eta_{k,i}^{\, n} \rvert_p = 1$ for
    $1 \le i \le d_k$ for almost all~$p \in \mathbb{P}$.
  \end{enumerate}

  \smallskip
  \rm{(c)} Suppose that, for each $k \in \{1, \ldots, c\}$, the cardinality of the set
  $$
    \mathbf{P}_k:=\left\{p\in \mathbb{P}\colon \prod_{i \not\in I_k(p)} \lvert \xi_{k,i}^{\, n}
      - \eta_{k,i}^{\, n} \rvert_p \neq 1\right\}
  $$
  is finite  and that $| \xi_{k,i} |_\infty \neq |\eta_{k,i}|_\infty$ for $1 \le i \le d_k$.

  Then the growth rate of the  Reidemeister coincidence numbers of iterations exists and is given by
  \begin{align}\begin{split}\label{eq:growth-coinc}
    R^{\infty}(\varphi, \psi) &= R^{\infty}_{+}(\varphi, \psi) = R^{\infty}_{-}(\varphi, \psi) = \\
    &= \prod_{k=1}^c \prod_{p \in \mathbf{P}_k} \prod_{i \in S^*_k(p)} \max \{ \lvert \xi_{k,i} \rvert_p,  |\eta_{k,i}|_p\} \cdot \\
    & \hspace{3em}
     \cdot \prod_{k=1}^c \prod_{p \in \mathbf{P}_k} \, \prod_{i \in S_k(p)\setminus S^*_k(p)}\lvert \eta_{k,i} \rvert_p  \cdot
     \prod_{k=1}^c \prod_{i=1}^{d} \max\{|\xi_{k,i}|_\infty, |\eta_{k,i}|_\infty\},
  \end{split}\end{align}
  where for $p \in \mathbf{P}_k$ we write $S_k(p) = \{1,\dots,d_k\}\setminus I_k(p)$
  and $S_k^*(p) = \big\{i \in S_k(p) \mid \lvert \xi_{k,i} \rvert_p \neq |\eta_{k,i}|_p \big\}$
\end{theorem}

\begin{proof}
  The equation in (a) is a generalization of the Roman'kov formula for ordinary twisted conjugacy classess
  in finitely generated nilpotent groups, see \cite[Prop.~2.1]{FelsKlop}.
  Therefore in what follows we consider the Reidemeister numbers on factor groups $G_k$.

  The profinite completions of the torsion-free abelian group $G_k$ and its
  endomorphisms decompose as direct products
  \[
    \iota \colon G_k \to \widehat{G}_k = \prod_{p \in \mathbb{P}} \widehat{G}_{k,p}
    \qquad \text{and} \qquad
    \widehat{\varphi}_k = \prod_{p \in \mathbb{P}} \widehat{\varphi}_{k,p},
    \quad \widehat{\psi}_k = \prod_{p \in \mathbb{P}} \widehat{\psi}_{k,p},
 \]
 where, for each prime $p$, the Sylow pro-$p$ subgroup of
 $\widehat{G}_k$ constitutes the pro-$p$ completion $\widehat{G}_{k,p}$
 of~$G_k$, equipped with endomorphisms
 $\widehat{\varphi}_{k,p}, \widehat{\psi}_{k,p} \colon \widehat{G}_{k,p} \to
 \widehat{G}_{k,p}$.
For the Reidemeister coincidence numbers we have a property analogous to Lemma~\ref{lem:prof-compl}
(see \cite{FelsKlop}) and as a result
 \[
 R(\varphi_k^n,\psi_k^n) = \prod_{p \in \mathbb{P}}
 R(\widehat{\varphi}_{k,p}^{\, n}, \widehat{\psi}_{k,p}^{\, n}), \qquad
 \text{for $n \in \mathbb{N}$;}
 \]
 in particular, $R(\varphi_k^n,\psi_k^n) < \infty$ implies that
 $R(\widehat{\varphi}_{k,p}^{\, n}, \widehat{\psi}_{k,p}^{\, n}) = 1$ for
 almost all $p \in \mathbb{P}$ so that the product is only formally
 infinite.

 We can basically repeat the reasoning from the Theorem \ref{thm:main-result-nilpotent} for both endomorphisms $\varphi_k, \psi_k$ simultaneously
 which leads us to the equality
 \[
  R(\widehat{\varphi}_{k,p}^{\, n},\widehat{\psi}_{k,p}^{\, n}) = \lvert
  \det(- \widehat{\varphi}_{k,p}^{\, n} + \widehat{\psi}_{k,p}^{\, n})
  \rvert_p^{\, -1} = \prod_{i \in I_k(p)} \lvert \xi_{k,i}^{\, n} -
  \eta_{k,i}^{\, n} \rvert_p^{\, -1},
\]
 which gives us (\ref{equ:key-formula-2}).

 To prove (\ref{eq:growth-coinc}) we follow the same logic as in the proof of the Theorem \ref{thm:main-result-nilpotent}
 being careful when substituting eigenvalues $\eta_{k,i}$ related to iterations of the second endomorphism $\psi$.
  For $p \in \mathbf{P}_k$ we write
  \[
    S_k(p) = \{1,\ldots,d_k\} \setminus I_k(p) \qquad \text{and} \qquad
    S^*_k(p) = \{i \in S_k(p) \mid \lvert \xi_{k,i} \rvert_p \neq \lvert
    \eta_{k,i} \rvert_p\};
  \]
  we remark right away that $\eta_{k,i} \neq 0$ for every
  $i \in S_k(p) \setminus S^*_k(p)$, because otherwise
  $\xi_{k,i} = \eta_{k,i} = 0$ and $\varphi_{k,\mathbb{Q}} - \psi_{k,\mathbb{Q}}$
  would have rank less than $d_k$.
  We set
  \[
    b := \prod_{k=1}^c \prod_{p \in \mathbf{P}_k} \, \prod_{i \in S^*_k(p)} \max \{ \lvert
    \xi_{k,i} \rvert_p, \lvert \eta_{k,i} \rvert_p \}, \qquad
    \eta := \prod_{k=1}^c \prod_{p \in \mathbf{P}_k} \, \prod_{i \in S_k(p)\setminus
      S^*_k(p)}\lvert \eta_{k,i} \rvert_p
  \]
  and, for $n \in \mathbb{N}$,
  \begin{gather*}
    g(n) := b^n\cdot\eta^n \cdot \prod_{k=1}^c \prod_{i=1}^{d_k} \lvert \xi_{k,i}^{\, n} - \eta_{k,i}^{\, n}
    \rvert_\infty \\
    f(n) := \prod_{k=1}^c \prod_{p \in \mathbf{P}_k} \, \prod_{i \in S_k(p)\setminus
      S_k^*(p)} \big\vert (\xi_{k,i} \eta_{k,i}^{\, -1})^n - 1 \big\vert_p.
  \end{gather*}
  We use (\ref{equ:key-formula-2}) along with basic properties of non-archimedean norm and we factor out $\eta$ to get
  \begin{equation}\label{eq:reidemeister-coincidence-g-f}
     R(\varphi^n, \psi^n) = g(n) \cdot f(n).
  \end{equation}

  Now we open up the product $\prod_{i=1}^{d_k} |\xi_{k,i}^{\, n} - \eta_{k,i}^{\, n}|_\infty$
  for every $k=1,\dots,c$ in $g(n)$.
  Complex eigenvalues $\xi_{k,i}, \eta_{k,i}$ in the spectra of $\varphi_{k,\Q}, \psi_{k,\Q}$ respectively,
  appear in pairs with their complex conjugates $\overline{\xi_{k,i}}, \overline{\eta_{k,i}}$ respectively.
  Such pairs can be lined up with one another in a simultaneous triangularisation as follows.
  Write $\varphi_{L_k}, \psi_{L_k}$ for the induced endomorphisms of the $L$-vector space
  $V := L \otimes_Q G \cong L^k \hookrightarrow \C^d$.
  If $v \in V$ is, at the same time, an eigenvector of $\varphi_{L_k}$ with complex eigenvalue $\xi_{k,d}$
  and an eigenvector ofr $\psi_{L_k}$ with eigenvalue $\eta_{k,d}$,
  then there is $w \in V$ such that $w$ is, at the same time, an eigenvector of $\varphi_{L_k}$
  with eigenvalue $\overline{\xi_{k,d}} \neq \xi_{k,d}$ and an eigenvector of $\psi_{L_k}$
  with eigenvalue $\overline{\eta_{k,d}}$, possibly equal to $\eta_{d,k}$.
  We can start the complete flag of $\{\varphi, \psi\}$-invariant subspaces of $V$ with
  $\{0\} \subset \langle v \rangle \subset \langle v, w \rangle$ and proceed with
  $V / \langle v, w \rangle$ by induction.
  We treat complex eigenvalues of $\psi_{L_k}$ in the same way.

  If at least one of $\xi_{k,i}, \eta_{k,i}$ is complex so that they are paired with
  $\xi_{k,j} = \overline{\xi_{k,i}}, \eta_{k,j} = \overline{\eta_{k,i}}$ for suitable $j \neq i$
  as discussed above, we see that
  \[
    |\xi_{k,i}^{\, n} - \eta_{k,i}^{\, n}|_\infty \cdot |\xi_{k,j}^{\, n} - \eta_{k,j}^{\, n}|_\infty
    = |\xi_{k,i}^{\, n} - \eta_{k,i}^{\, n}|_\infty^2
    = (\xi_{k,i}^{\, n} - \eta_{k,i}^{\, n}) \cdot (\overline{\xi_{k,j}}^{\, n} - \overline{\eta_{k,j}}^{\, n})
  \]
  If $\xi_{k,i}, \eta_{k,i}$ are both real eigenvalues, not paired up with another pair of eigenvalues,
  then we have
  \begin{gather*}
    |\xi_{k,i}^{\, n} - \eta_{k,i}^{\, n}|_\infty = \delta_{1,k,i}^n - \delta_{2,k,i}^n,\\
    \delta_{1,k,i} := \max \{|\xi_{k,i}|_\infty, |\eta_{k,i}|_\infty\}, \quad
    \delta_{2,k,i} := \frac{\xi_{k,i} \cdot \eta_{k,i}}{\delta_{1,k,i}}\\
  \end{gather*}
  We can now write the product $\prod_{i=1}^{d_k} |\xi_{k,i}^{\, n} - \eta_{k,i}^{\, n}|_\infty$
  in a manner analogous to \eqref{eq:form_of_w} which leads us to
  \[
    g(n) = \sum_{j \in J} c_j w_j^n,
  \]
  where $J$ is a finite index set, $c_j = \pm 1$ and $w_j \in \C \setminus \{0\}$.
  There is a dominant term $w_m$ for which
  \[
    |w_m|_\infty =  b \cdot \eta \cdot
    \prod_{k=1}^c \prod_{i=1}^{d_k} \max \big\{ |\xi_{k,i}|_\infty, |\eta_{k,i}|_\infty \big\}.
  \]

  We are left to show that $\lim_{n\to\infty} \sqrt[n]{f(n)} = 1$.
  First of all, note that $\xi_{k,i} \neq \eta_{k,i}$ because otherwise $\varphi_{k,\Q} - \psi_{k,\Q}$
  would have rank less than $d_k$, hence $\xi_{k,i}\cdot\eta_{k,i}^{-1} \neq 1$.
  By construction of $S_k(p) \setminus S_k^*(p)$ we have $|\xi_{k,i}|_p = |\eta_{k,i}|_p$,
  hence $|\xi_{k,i} \cdot \eta_{k,i}^{-1}|_p = 1$.
  Now we can proceed exactly as in the proof of Theorem \ref{thm:main-result-nilpotent}
  by deriving bounds analogous to \eqref{eq:fn_bounds}.

  Finally, from \eqref{eq:reidemeister-coincidence-g-f} we get
  $R^\infty(\varphi, \psi) = |w_m|_\infty$ which is (\ref{eq:growth-coinc}).
\end{proof}

\begin{remark}\label{thm:growth-rate-coincidence-reidemeister-finitely-generated}
  In the same manner as in the remark \ref{thm:growth-reidemeister-finitely-generated}
  if the group $G$ is actually finitely generated then the $p$-adic part
  in the product {\rm (\ref{eq:growth-coinc})} disappears.
  The explicit formula for the growth rate of Reidemeister coincidence numbers becomes
  \begin{equation}\label{eq:growth-rate-coincidence-reidemeister-finitely-generated}
    R^\infty(\varphi,\psi) =
    \prod_{k=1}^c \prod_{i=1}^{d} \max\{|\xi_{k,i}|_\infty, |\eta_{k,i}|_\infty\}
  \end{equation}
\end{remark}

\bigskip
\section{The Gauss congruences  for  the sequence   $\{R(\varphi^n,\psi^n)\}$ of the   Reidemeister coincidence numbers }\label{sec:congruences}

Let $\mu(n)$, $n\in\N$, be the {\em M\"obius function},
i.e.
$$
\mu(n) =
\left\{
\begin{array}{ll}
1 & \mbox{ if } n=1,  \\
(-1)^k & \mbox{ if }  n  \mbox{ is a product of } k \mbox{ distinct primes,}\\
0 & \mbox{ if }  n  \mbox{ is not square-free.}
\end{array}
\right.
$$
In number theory, the following Gauss congruence for integers holds:
$$
\sum_{d\mid n}\mu(n/d)\cdot a^{d}\equiv 0\mod n
$$
for any integer $a$ and any natural number $n$. In the case of a prime power $n=p^r$, the Gauss congruence turns into the Euler congruence. Indeed, for $n=p^r$ the M\"obius function $\mu(n/d)=\mu(p^r/d)$ is nonzero only in two cases: when $d=p^r$ and when $d=p^{r-1}$. Therefore, from the Gauss congruence we obtain the Euler congruence
$$
a^{p^r}\equiv a^{p^{r-1}}\mod p^r.
$$
When $(a,n)=1, n=p^r$, these congruences are equivalent to the following classical Euler's theorem:
$$
a^{\varphi(n)}\equiv 1\mod n.
$$
These congruences have been generalized from integers to some other mathematical invariants such as the traces of powers of all integer matrices $A$ and the Lefschetz numbers of iterations  of a map {see \cite{mp99,Z}}:
\begin{align}
\label{Gauss}
&\sum_{d\mid n}\mu(n/d)\cdot \tr(A^{d})\equiv 0\mod n,\\
\label{Euler}
&\tr(A^{p^r})\equiv\tr(A^{p^{r-1}})\mod p^r.
\end{align}
{A. Dold in \cite{Dold83} (see also \cite[Theorem 3.1.4]{JezMar}) proved by a geometric argument the following congruences \eqref{Dold} for the Lefschetz numbers of iterations of a map $f$ on a compact ANR $X$ and any natural number $n$

\begin{align}
\label{Dold}
\sum_{d\mid n}\mu(n/d)\cdot L(f^{d})\equiv 0\mod n .\tag{DL}
\end{align}
 These congruences are now called the Dold congruences.} It is also shown in \cite{mp99} (see also \cite[Theorem~9]{Z}) that the above congruences \eqref{Gauss}, \eqref{Euler} and \eqref{Dold} are equivalent.
 For more detailed discussion see \cite{ByszGraffWard}.

The following Gauss congruences  for the  Reidemeister numbers of iterations of a tame automorphism $\varphi$ of polycyclic-by-finite group are proven in \cite{polyc}
$$
\sum_{d\mid n}\mu(n/d)\cdot R(\varphi^{d})\equiv 0\mod n.
$$

It turns out that the Gauss congruences for Reidemeister coincidence numbers holds if
the Reidemeister coincidence zeta function is rational.

The \emph{Reidemeister coincidence zeta function} for a tame pair
$(\varphi, \psi)$ of  endomorphisms  of a group $G$
was defined in \cite{FelsKlop} as
\begin{equation}\label{eq:zeta-def}
  R_{\varphi, \psi}(z) := \exp \left( \sum_{k=1}^{\infty}\frac{R(\varphi^k, \psi^k)}{k} z^k \right).
\end{equation}
In this section we use the following  criterion for the rationality of the Reidemeister coincidence zeta function
for a  tame pair of  endomorphisms  of a  torsion-free nilpotent group of finite Pr\"ufer rank to
prove the Gauss congruences  for the Reidemeister coincidence numbers of iterations of these   endomorphisms.

\begin{theorem}[\protect{\cite[Theorem 1.4 (3)]{FelsKlop}}]\label{thm:reidemeister-rationality-klopsch}
  Let $\varphi, \psi$ be as in \rm{Theorem \ref{thm:growth-coinc} (c)}.
  Then the Reidemeister coincidence zeta function $R_{\varphi, \psi}(z)$ is either a rational function
  or it has a natural boundary as its radius of convergence.
  Furthermore, the latter occurs iff $|\xi_{k,i}|_p = |\eta_{k,i}|_p$ for some $k \in \left\{ 1, \dots, c \right\},
  p \in \mathbf{P}_k, i \notin I_k(p)$.
\end{theorem}

To prove Gauss congruences under assumption of rationality of the Reidemeister coincidence zeta function
we follow \cite[Theorem~2.1]{BaBo} and \cite[Theorem 3.1.23, Proposition~3.1.12]{JezMar}.

\begin{lemma}\label{thm:rational-zeta-coefficients-sum}
  $R_{\varphi, \psi}(z)$ is rational iff for every $k \in \N$ we have
  \begin{equation}\label{eq:coincidence-sum}
    R(\varphi^k, \psi^k) = \sum_{i=1}^{r} \chi_i \lambda_i^k
  \end{equation}
  where $r \in \N, \ \chi_i \in \Z$ and $\lambda_i \in \C$ are distinct algebraic integers.
\end{lemma}
\begin{proof}
  $\Leftarrow$)
  Let us split all the $\chi_i$ into two sets according to their signs.
  We have $r^+$ positive $\chi_i^+ = \chi_i$ and $r^-$ positive $\chi_i^- = -\chi_i$ where $r^+ + r^- = r$.
  Denote $\lambda_i$ as $\lambda_i^+$ or $\lambda_i^-$ with respect to the sign of corresponding $\chi_i$.
  With this notation we have
  \[
    \sum_{i=1}^{r} \chi_i \lambda_i^k \ = \ \sum_{i=1}^{r^+} \chi_i^+ {\lambda_i^+}^k \ - \ \sum_{i=1}^{r^-} \chi_i^- {\lambda_i^-}^k.
  \]
  Using this equality in (\ref{eq:coincidence-sum}), properties of $\exp$
  and the identity $\ln(1/(1-t)) = \sum_{n=1}^\infty t^n/n$ we get
  \[
    R_{\varphi, \psi}(z) =
    \frac{\prod_{i=1}^{r^+} \left( \exp \left( \sum_{k=1}^\infty \frac{(\lambda_i^+ z)^k}{k} \right) \right)^{\chi_i^+}} {\prod_{i=1}^{r^-} \left( \exp \left( \sum_{k=1}^\infty \frac{(\lambda_i^- z)^k}{k} \right) \right)^{\chi_i^-}} =
    \frac{\prod_{i=1}^{r^-} \left( 1-\lambda_i^- z \right)^{\chi_i^-}} {\prod_{i=1}^{r^+} \left( 1-\lambda_i^+ z \right)^{\chi_i^+}}
  \]
  and the zeta function is rational.

  $\Rightarrow$)
  Assume $R_{\varphi, \psi}(z)$ is rational.
  By definition $R_{\varphi, \psi}(0) = 1 \neq 0$ therefore we can write
  \[
    R_{\varphi, \psi}(z) = \frac{\prod_i (1 - \beta_i z)}{\prod_j (1 - \gamma_j z)}
  \]
  where $\beta_i, \gamma_j$ are non-zero roots of polynomials and the products are finite.
  Combine the roots into one sequence $\{\lambda_i\}$ where for $i \neq j$ there is $\lambda_i \neq \lambda_j$.
  Taking multiplication into account we can write
  \begin{equation}\label{eq:zeta-product}
    R_{\varphi, \psi}(z) = \prod_i \left( 1 - \lambda_i z \right)^{\chi_i}, \quad \chi_i \in \Z.
  \end{equation}

  Let us consider the logarithmic derivative
  \begin{equation}\label{eq:S-def}
    S(z) := \frac{d}{dz} \ln R_{\varphi, \psi}(z) = \sum_{k=1}^\infty R(\varphi^k, \psi^k) \cdot z^{k-1}.
  \end{equation}
  From (\ref{eq:zeta-product}) we have
  \begin{align}\label{eq:S=rational}
  \begin{split}
    S(z) &= \frac{d}{dz} \left( \sum_i \chi_i \ln(1-\lambda_i z) \right)
    \ = \ \sum_i \chi_i \frac{-\lambda_i}{1-\lambda_i z} \\
    &= \sum_i (-\chi_i \lambda_i) \sum_{k=0}^\infty \lambda_i^k z^k
    \ = \ \sum_{k=1}^\infty \left( \sum_i -\chi_i \lambda_i^k \right) z^{k-1}.
  \end{split}
  \end{align}
  Comparing coefficients at $z^{k-1}$ we get $R(\varphi^k, \psi^k) = \sum_i \chi_i \lambda_i^k$
  up to the signs of $\chi_i$,
  the~sum is finite, $\chi_i \in \Z$ and $\lambda_i$ are algebraic numbers.

  We are going to prove that $\lambda_i$ are actually algebraic integers.
  From (\ref{eq:S-def}) we know that $S(z)$ is given by a power series with integer coefficients
  and from (\ref{eq:S=rational}) that it is a rational function.
  Therefore $S(z) = u(z) / v(z)$, where $u, v \in \Q[z]$.
  Fatou lemma (see \cite[Lemma~3.1.31]{JezMar}) says that $u(z), v(z)$ in this setting have the forms
  \[
    u(z) = \sum_{i=0}^s a_i z^i, \quad v(z) = 1 + \sum_{j=1}^q b_j z^j
  \]
  where the coefficients $a_i, b_j \in \Z$.
  Comparing this with (\ref{eq:S=rational}) we get
  \[
    \frac{\sum_{i=0}^s a_i z^i}{1 + \sum_{j=1}^q b_j z^j} \ = \ S(z) \
    = \ \sum_i \frac{-\chi_i\lambda_i}{1-\lambda_i z}
  \]
  and it is easy to check that all $\lambda_i$ are roots of the polynomial
  $\tilde{v}(z) = z^q + \sum_{j=1}^q b_j z^{q-j}$.
  Hence they are algebraic integers.
\end{proof}

\begin{lemma}\label{thm:coefficients-sum-map-bouquet}
  There exists a map $f: X \to X$ of a compact euclidean neighborhood retract such that
  $L(f^k) = R(\varphi^k, \psi^k)$ for every $k \in \N$ iff
  the equality \rm{\eqref{eq:coincidence-sum}} holds for every $k \in \N$
\end{lemma}
\begin{proof}
  $\Rightarrow$)
  Denote by $H_i(f): H_i(X;\Q) \to H_i(X;\Q)$ the map induced on $i$-th homology space and put
  \[
    A_e := \bigoplus_{i-even} H_i(f), \quad A_o := \bigoplus_{i-odd} H_i(f).
  \]
  With this notation we can write
  \[
    R(\varphi^k, \psi^k) = L(f^k) = \tr A_e^k - \tr A_o^k.
  \]
  Let $\lambda_i^+, \lambda_j^-$ be all distinct eigenvalues of $A_e, A_o$ respectively,
  each of multiplicity $\chi_i^+, \chi_j^-$ respectively.
  Then
  \[
    R(\varphi^k, \psi^k) \ = \ \sum_{i} \chi_i^+ {\lambda_i^+}^k \ - \ \sum_j \chi_j^- {\lambda_j^-}^k,
  \]
  where $\chi_i^+, \chi_j^- \in \Z$ and all $\lambda_i^+, \lambda_j^-$ are distinct algebraic integers
  (as roots of characteristic polynomials of integer matrices).

  $\Leftarrow$)
  Assume that for every $k \in \N$ the equality $$R(\varphi^k, \psi^k) = \sum_{i=1}^{r} \chi_i \lambda_i^k$$
  where $r \in \N, \ \chi_i \in \Z$ and $\lambda_i \in \C$ are distinct algebraic integers, holds.
  It turns out that if $\lambda_i, \lambda_j$ are algebraically conjugate then $\chi_i = \chi_j$.
  Indeed, denote by $v(z)$ an irreducible polynomial roots of which are $\lambda_i, \lambda_j$.
  Denote by $\Sigma$ the field of $v(z)$ and let $\sigma$ be an element of the Galois group of $\Sigma$.
  By definition $\sigma$ is the identity function on $\Q$.
  It is well-known that $\sigma$ acts as a~permutation of the set of roots of $v(z)$,
  denote the induced permutation of indices by $\sigma$ as well,
  i.e. $\sigma(\lambda_i) = \lambda_{\sigma(i)}$.
  We get
  \[
    \sigma \left( R(\varphi^k, \psi^k) \right)
    = \sigma \left( \sum_{i=1}^r \chi_i \lambda_i^k \right)
    = \sum_{i=1}^r \chi_i \sigma(\lambda_i)^k
    = \sum_{i=1}^r \chi_i \lambda_{\sigma(i)}^k
    = \sum_{i=1}^r \chi_{\sigma^{-1}(i)} \lambda_i^k.
  \]
  On the other hand a Reidemeister coincidence number is an integer
  therefore $\sigma \left( R(\varphi^k, \psi^k) \right) = R(\varphi^k, \psi^k)$
  and for all $k \in \N$ we get
  \[
    \sum_{i=1}^r \chi_i \lambda_i^k \ = \ \sum_{i=1}^r \chi_{\sigma^{-1}(i)} \lambda_i^k.
  \]
  This equality holds especially for $k = 1, \dots r$, let us write it in a matrix form:
  \[
    \begin{pmatrix}
        \lambda_1   & \dots  & \lambda_r \\
        \vdots      & \ddots & \vdots    \\
        \lambda_1^r & \dots  & \lambda_r^r
    \end{pmatrix}
    \begin{pmatrix}
        \chi_1 \\ \vdots \\ \chi_r
    \end{pmatrix}
    =
    \begin{pmatrix}
        \lambda_1   & \dots  & \lambda_r \\
        \vdots      & \ddots & \vdots    \\
        \lambda_1^r & \dots  & \lambda_r^r
    \end{pmatrix}
    \begin{pmatrix}
        \chi_{\sigma^{-1}(1)} \\ \vdots \\ \chi_{\sigma^{-1}(r)}
    \end{pmatrix}.
  \]
  Since all $\lambda_i$ are distinct, the matrix is invertible (as the Vandermonde matrix),
  hence $\chi_i = \chi_{\sigma^{-1}(i)}$ for all $i$ and for arbitrary $\sigma$.
  We know that for every pair $\lambda_i, \lambda_j$ there is an automorphism $\sigma$
  such that  $\sigma(\lambda_i) = \lambda_j$.
  It means that $\chi_i = \chi_j$ whenever $\lambda_i, \lambda_j$ are algebraically conjugate.

  Now let $v(z) = \prod_{i=1}^r (\lambda_i - z)$ be a polynomial roots of which are all $\lambda_i$.
  Let
  \begin{align*}
    v(z) = \prod_{\alpha} v_{\alpha}(z)
    = \prod_{\alpha} \left( z^{r_\alpha} + b_1^{(\alpha)} z^{r_\alpha-1}  + \dots +
      b_{r_\alpha - 1}^{(\alpha)} z + b_{r_\alpha}^{(\alpha)} \right),
      \quad b_i^{(\alpha)} \in \Z
  \end{align*}
  be its decomposition into irreducible polynomials.
  We split the set $A = A^+ \cup A^-$ of indices $\alpha$ into two subsets according to the sign of $\chi_\alpha$.
  We can write
  \begin{align}\label{eq:R=sum-of-eigenvalues}
  \begin{split}
    R(\varphi^k, \psi^k)
    &= \sum_{\alpha \in A} \chi_\alpha \left( \sum_{i=1}^{r_\alpha} \left(\lambda^{(\alpha)}_i\right)^k \right) \\
    &= \sum_{\alpha \in A^+} \chi_\alpha \left( \sum_{i=1}^{r_\alpha} \left(\lambda^{(\alpha)}_i\right)^k \right)
      - \sum_{\alpha \in A^-} |\chi_\alpha| \left( \sum_{i=1}^{r_\alpha} \left(\lambda^{(\alpha)}_i\right)^k \right).
  \end{split}
  \end{align}
  Consider the integer matrix
  \begin{equation*}
    M_\alpha := \begin{pmatrix}
        0      & 0      & \dotsb & 0      & - b_{r_\alpha}^{(\alpha)} \\
        1      & 0      & \dotsb & 0      & - b_{r_{\alpha-1}}^{(\alpha)} \\
        \vdots & \vdots &        & \vdots & \vdots \\
        0      & 0      & \dotsb & 0      & - b_{2}^{(\alpha)} \\
        0      & 0      & \dotsb & 1      & - b_{1}^{(\alpha)} \\
    \end{pmatrix}
  \end{equation*}
  Straightforward computation using the Laplace expansion shows that $\det(zI-M_\alpha) = v_\alpha(z)$
  therefore $\{\lambda_i^{(\alpha)}\}$ are the eigenvalues of $M_\alpha$.
  Continuing (\ref{eq:R=sum-of-eigenvalues}) we write
  \begin{equation*}
     R(\varphi^k, \psi^k)
     = \sum_{\alpha \in A^+} \chi_\alpha \tr M_\alpha^k - \sum_{\alpha \in A^-} |\chi_\alpha| \tr M_\alpha^k
  \end{equation*}
  Denote
  \begin{equation*}
    A_e := \bigoplus_{\alpha \in A^+} \bigoplus_{i=1}^{\chi_\alpha} M_\alpha, \quad
    A_o := \bigoplus_{\alpha \in A^-} \bigoplus_{i=1}^{|\chi_\alpha|} M_\alpha.
  \end{equation*}
  Using this notation we finally have
  \[
    R(\varphi^k, \psi^k) = \tr A_e^k - \tr A_o^k.
  \]
  The thesis follows from the next Lemma \ref{thm:integer-matrix-map-bouquet}.
\end{proof}

The next lemma is a result of \cite[Theorem 2.1]{BaBo}.
We present a formulation more convenient for our purposes.

\begin{lemma}[\protect{\cite[Proposition 3.1.12]{JezMar}}]\label{thm:integer-matrix-map-bouquet}
  Let X be the bouquet of $n_1$ circles and $n_2$ \rm{2}-spheres, $n_1 \geq 2$.
  Then for each pair of matrices $A_e \in M_{n_2}(\Z), A_o \in M_{n_1-1}(\Z)$
  there exists a self-map $f: X \to X$ satisfying
  \[
    L(f^k) = \tr A_e^k - \tr A_o^k
  \]
  for every $k \in \N$.
\end{lemma}
\begin{proof}
  Let $X_2 := \bigvee_{1}^{n_2} S^2$ be the bouquet of $n_2$ 2-dimensional spheres
  and $X_1 := \bigvee_{1}^{n_1-1} S^1$ be the bouquet of $n_1-1$ circles.
  Then $X = S^1 \vee X_1 \vee X_2$.

  Let $f_2: X_2 \to X_2, \ f_1: X_1 \to X_1$ be the maps defined in the obvious way
  by the~matrices $A_e, A_o$ respectively and let $f_0 = \Id : S^1 \to S^1$ be the identity map.
  We claim that $f := f_0 \vee f_1 \vee f_2$ is the map satisfying the hypothesis.

  We will denote a map induced by $f$ on $i$-th homology space by $H_i(f): H_i(X, \Q) \to H_i(X, \Q)$.
  Note that we have
$$      H_2(f_2) = A_e, \quad H_1(f_2) = 0, \quad H_0(f_2) = 1,$$
$$      H_1(f_1) = A_o, \quad H_0(f_1) = 1,$$
$$      H_1(f_0) = 1,   \quad H_0(f_0) = 1.$$
  From the definition of the Lefschetz number and properties of the trace of a matrix
  \begin{align*}
      L(f^k) &= \sum_{i=0}^2 (-1)^i \Tr(H_i(f^k))\\
      &= 1 - (\Tr A_o^k + 1) + \Tr A_e^k\\
      &= \Tr A_e^k - \Tr A_o^k
  \end{align*}
  and the lemma is proved.
\end{proof}

\bigskip
Finally, the main result of this paragraph is

\begin{theorem}\label{thm:rational-congruences}
  Let $\varphi, \psi$ be a tame pair of endomorphisms of a torsion-free nilpotent group of finite Pr\"ufer rank.
  If the Reidemeister coincidence zeta function $R_{\varphi, \psi}(z)$ is rational then
  for every $n \in \N$ we have the Gauss congruences
  \begin{equation}\label{eq:congruence-coincidence}
    \sum_{k|n} \mu(n/k) R(\varphi^k, \psi^k) \equiv 0 \mod n.
  \end{equation}
\end{theorem}

\begin{proof}
  Lemma \ref{thm:rational-zeta-coefficients-sum} and Lemma \ref{thm:coefficients-sum-map-bouquet}
  show that $R_{\varphi, \psi}(z)$ is rational iff for every $k \in \N$ we have $L(f^k) = R(\varphi^k, \psi^k)$.
  Direct substitution in the Dold congruences \eqref{Dold} finishes the proof.
\end{proof}

\begin{example}
  Consider the group $\mathbb{Z}$ written additively and its two endomorphisms
  \[
    \varphi: x \mapsto d_\varphi x, \quad \psi: x \mapsto d_\psi x, \quad d_\varphi, d_\psi \in \mathbb{Z}.
  \]
  By definition elements $x,y \in \mathbb{Z}$ are in the same Reidemeister coincidence class
  of $(\varphi^k, \psi^k)$ if there exists $g \in \mathbb{Z}$ such that
  \begin{align*}
    y \ = \ d_\psi^k g + x - d_\varphi^k g \ = \ x + g(d_\psi^k - d_\varphi^k)
  \end{align*}
  therefore
  \begin{equation*}
    R(\varphi^k, \psi^k) = \begin{cases}
      | d_\psi^k - d_\varphi^k |, &d_\psi^k \neq d_\varphi^k,\\
      \infty, &\text{otherwise}
    \end{cases}.
  \end{equation*}
  Therefore the pair $(\varphi, \psi)$ is tame iff $|d_\varphi| \neq |d_\psi|$.
  Let the pair be tame in what follows.
  Denote $d_1 = \max\left( |d_\varphi|, |d_\psi| \right), d_2 = d_\varphi d_\psi / d_1$.
  With this notation $|d_\psi^k - d_\varphi^k| = d_1^k - d_2^k$ and it is easy to compute
  the Reidemeister coincidence zeta function
  \begin{align*}
    R_{\varphi, \psi}(z) &= \exp \left( \sum_{k=1}^\infty \frac{R(\varphi^k, \psi^k)}{k} z^k \right)
    = \exp \left( \sum_{k=1}^\infty \frac{d_1^k - d_2^k}{k} z^k \right) \\
    &= \exp \big( \ln(1-d_2z) - \ln(1-d_1z) \big) = \frac{1-d_2z}{1-d_1z}
  \end{align*}
  which turns out to be rational.
  Theorem \ref{thm:rational-congruences} states that (\ref{eq:congruence-coincidence}) holds in this case.
  It can be also easily seen directly because
  \begin{align*}
    \sum_{k|n} \mu(n/k) R(\varphi^k, \psi^k) \ = \ \sum_{k|n} \mu(n/k) d_1^k - \sum_{k|n} \mu(n/k) d_2^k
    \ \equiv \ 0 - 0 \mod n
  \end{align*}
  by the original Gauss congruences for integers $d_1, d_2$.
\end{example}

\bigskip
\section{Nielsen coincidence zeta function, the growth rate  and the Gauss congruences for  the sequence  $\{N(f^n, g^n)\}$}\label{sec:coincidence-nielsen}

Let $f, g: X \to Y$ be two maps between two compact connected manifolds.
Maps $f,g$ induce homomorphisms $f_\#, g_\#$ between fundamental groups of $X$ and $Y$.
In section \ref{sec:intro} we have defined the \emph{topological} Reidemeister coincidence number of maps $R(f,g) := R(f_\#,g_\#)$
to be the Reidemeister coincidence number of the~induced homomorphisms.
If $Y=X$ then we say that the pair of maps $f,g$ is \emph{tame} if
$R(f^n, g^n) < \infty$ for every $n \in \N$.
We can simply reword the definition of growth rate of the sequence \{$R(f^n, g^n)\}$
of the Reidemeister coincidence numbers
for the pair of maps $(f,g)$: $R^\infty(f,g) := \lim_{n \to \infty} R(f^n, g^n)^{1/n} = R^\infty(f_\#,g_\#)$.

In the same manner as with the Reidemeister numbers we define the growth rate of the sequence
\{$N(f^n, g^n)\}$  of the Nielsen coincidence numbers
(if it exists)
\[
  N^\infty(f,g) := \lim_{n \to \infty}  N(f^n, g^n)^{1/n}.
\]

In this section we are going to work with compact nilmanifolds.
Let $G$ be a Lie group and let $\G \subset G$ be its discrete subgroup.
$\G$ is a \emph{lattice} if there exists a finite $G$-invariant measure on $G/\G$.
A compact manifold $X$ of dimension $d$ is a \emph{nilmanifold} if it is a quotient space $G/\G$
of a simply-connected nilpotent Lie group $G$ by a lattice $\G$ of rank $d$ (see \cite[Definition 6.2.5]{JezMar}).

We are going to make use of the following result which shows connections between different coincidence numbers
under certain assumptions.
For a definition of the Lefschetz coincidence number $L(f,g)$ see for example \cite{Vick}.

\begin{theorem}[\protect{\cite[Theorem 2.3]{Gonc}}]\label{thm:nielsen-reidemeister}
  Let $f,g:N \to N$ be a pair of maps where $N$ is a compact nilmanifold.
  Then the following conditions are equivalent:
  \begin{enumerate}
    \item[(i)] $N(f,g) \neq 0$;
    \item[(ii)] $\Coin(f_\#, g_\#) = 1$;
    \item[(iii)] $R(f,g) < \infty$.
  \end{enumerate}
  If one of the conditions holds, then
  \[
    R(f,g) = N(f,g) = |L(f,g)|.
  \]
\end{theorem}

The theorem above along with other results of this paper motivate us to define and study
the \emph{Nielsen coincidence zeta function} of a pair $f,g: N \to N$ of maps
of a compact polyhedron to itself
\begin{equation}\label{eq:nielsen-zeta-def}
  N_{f,g}(z) := \exp \left( \sum_{n=1}^{\infty} \frac{N(f^n, g^n)}{n} z^n \right).
\end{equation}
Because of finitness of Nielsen coincidence numbers,
the Nielsen coincidence zeta function is always well-defined.
If the pair $f,g$ is not tame then the Reidemeister coincidence zeta function is not defined
therefore the Nielsen coincidence zeta function is a different object in general.

We state our main result regarding Nielsen coincidence zeta function.
\begin{theorem}\label{thm:main-result-nielsen-coincidence}
  Let $f,g:N \to N$ be a tame pair of maps of a compact nilmanifold to itself.
  Denote by $f_\#, g_\#:G \to G$ endomorphisms induced on the fundamental group of $N$.
  $G$ is a finitely generated torsion-free nilpotent group.
  Let $c$ denote the nilpotency class of $G$ and for $1 \leq k \leq c$ let
  ${f_\#}_k, {g_\#}_k: G_k \to G_k$ denote the tame pairs of induced endomorphisms
  of the finitely generated torsion-free abelian factor groups
  \[
    G_k = \sqrt[G]{\gamma_k(G)} / \sqrt[G]{\gamma_{k+1}(G)} \cong \Z^{d_k}
  \]
  for some $d_k \in \N$, that arise from the isolated lower central series of $G$.
  Every $G_k$ is a fundamental group of the $d_k$-dimensional torus $\T^{d_k}$ which is also a compact nilmanifold.
  Therefore every endomorphism ${f_\#}_k, {g_\#}_k$ can be realized by a selfmap
  $f_k, g_k: \T^{d_k} \to \T^{d_k}$.

  Then the following hold.
  \begin{enumerate}
    \item[(i)] For each $n \in \N$
      \begin{equation}\label{eq:nielsen-romankov}
        N(f^n, g^n) = \prod_{k=1}^c N(f_k^n, g_k^n)
      \end{equation}
    \item[(ii)] For $1 \leq k \leq c$ let
      \[
        {f_\#}_{k,\Q}, {g_\#}_{k,\Q}: G_{k,\Q} \to G_{k,\Q}
      \]
      denote the extensions of ${f_\#}_k, {g_\#}_k$ to the divisible hull
      $G_{k,\Q} = \Q \otimes G_k \cong \Q^{d_k}$.
      Suppose each pair of endomorphisms ${f_\#}_{k,\Q}, {g_\#}_{k,\Q}$ is simultaneously triangularizable.
      Let $\xi_{k,1}, \ldots, \xi_{k,d_k}$ and
      $\eta_{k,1}, \ldots, \eta_{k,d_k}$ be the eigenvalues of
      ${f_\#}_{k,\Q}$ and ${g_\#}_{k,\Q}$ in the field $\C$, including
      multiplicities, ordered so that, for $n \in \mathbb{N}$, the
      eigenvalues of ${f_\#}_{k,\Q}^{\, n} - {g_\#}_{k,\Q}^{\, n}$ are
      $\xi_{k,1}^{\, n} - \eta_{k,1}^{\, n}, \ldots, \xi_{k,d_k}^{\, n} - \eta_{k,d_k}^{\,n}$.
      Then for each $n \in \mathbb{N}$
       \begin{equation} \label{eq:nielsen-product-of-eigenvalues}
         N(f_k^n,g_k^n)
         = \prod_{i=1}^{d_k }\lvert \xi_{k,i}^{\, n} - \eta_{k,i}^{\, n} \rvert
       \end{equation}
     \item[(iii)] Suppose that $|\xi_{k,i}| \neq |\eta_{k,i}|$
       for $1 \leq k \leq c, \ 1 \leq i \leq d_k$.
       Then the Nielsen coincidence zeta function $N_{f,g}(z)$ is rational.
     \item[(iv)] With the same assumption $|\xi_{k,i}| \neq |\eta_{k,i}|$ for $1 \leq k \leq c, \ 1 \leq i \leq d_k$
       the growth rate of Nielsen coincidence numbers exists and can be computed using the explicit formula
       \begin{equation}\label{eq:nielsen-growth-formula}
         N^\infty(f,g) =
         \prod_{k=1}^c \prod_{i=1}^{d_k} \max\{|\xi_{k,i}|, |\eta_{k,i}|\}.
       \end{equation}
     \item[(v)] With the same assumption $|\xi_{k,i}| \neq |\eta_{k,i}|$ for $1 \leq k \leq c, \ 1 \leq i \leq d_k$
       we have the Gauss congruences
       \begin{equation}\label{eq:nielsen-gauss-congruences}
         \sum_{k|n} \mu(n/k)\cdot N(f^k, g^k) \equiv 0 \mod n
       \end{equation}
       for all $n \in \N$.
  \end{enumerate}
\end{theorem}
\begin{proof}
  Assumption of tameness of the pair $(f,g)$ gives us by definition
  the condition (iii) in Theorem \ref{thm:nielsen-reidemeister} for all iterations $f^n, g^n$.
  Therefore for all $n \in \N$ we have
  \begin{gather}\label{eq:nielsen=reidemeister}
  \begin{gathered}
    N(f^n, g^n) = R(f^n, g^n) = R(f_\#^n, g_\#^n), \\
    N(f_k^n, g_k^n) = R(f_k^n, g_k^n) = R({f_\#}_k^n, {g_\#}_k^n).
  \end{gathered}
  \end{gather}
  From now on we can follow arguments from the proof of Theorem \ref{thm:growth-coinc}.

  To prove (i) it is sufficient to make use of the Roman'kov formula (\ref{eq:reidemeister-coincidence-romankov})
  for endomorphisms $f_\#, g_\#$. We get
  \[
    N(f^n, g^n) = R(f_\#^n, g_\#^n)
    = \prod_{k=1}^c R({f_\#}_k^n, {g_\#}_k^n) = \prod_{k=1}^c N(f_k^n, g_k^n).
  \]

  To prove (ii) observe that the endomorphisms ${f_\#}_k, {g_\#}_k: G_k \to G_k \cong \Z^{d_k}$
  are represented by integer matrices $A_k, B_k \in M_{d_k}(\Z)$ associated to them respectively.
  There are unimodular matrices $M_{n,k}, N_{n,k}$ and a diagonal integer matrix
  $C_{n,k} = \diag(c_{n,1},\dots,c_{n,d_k})$
  such that $C_{n,k} = M_{n,k}(A_k^n - B_k^n)N_{n,k}$ and $\det C_{n,k} = \det(A_k^n - B_k^n)$.
  The order of the cokernel of ${f_\#}_k^n - {g_\#}_k^n$ is the order of the group
  $\Z/c_{n,1}\Z \oplus \dots \oplus \Z/c_{n,d_k}\Z$ therefore we can write
  \begin{align}
  \begin{split}
    R({f_\#}_k^n, {g_\#}_k^n) &= |\Coker({f_\#}_k^n - {g_\#}_k^n)| = |c_{n,1} \dots c_{n,d_k}| \\
    &= |\det({f_\#}_k^n - {g_\#}_k^n)| = |\det({f_\#}_{k,\Q}^n - {g_\#}_{k,\Q}^n)| \\
    &= \prod_{i=1}^{d_k} \lvert \xi_{k,i}^{\, n} - \eta_{k,i}^{\, n} \rvert \\
  \end{split}
  \end{align}
  Now using (\ref{eq:nielsen=reidemeister}) we get (\ref{eq:nielsen-product-of-eigenvalues}).

  To prove (iii) we need to open up the modules in (\ref{eq:nielsen-product-of-eigenvalues}).
  Complex eigenvalues $\xi_{k,i}, \eta_{k,i}$ in the spectrum of ${f_\#}_{k,\mathbb{Q}}, {g_\#}_{k,\mathbb{Q}}$
  appear in pairs with their complex conjugate $\ov{\xi_{k,i}}, \ov{\eta_{k,i}}$ respectively.
  If at least one of $\xi_{k,i}, \eta_{k,i}$ is complex (has nonzero imaginary part) we have
  \[
    |\xi_{k,i}^n - \eta_{k,i}^n| \cdot |\ov{\xi_{k,i}}^n - \ov{\eta_{k,i}}^n|
    = |\xi_{k,i}^n - \eta_{k,i}^n|^2
    = (\xi_{k,i}^n - \eta_{k,i}^n) \cdot (\ov{\xi_{k,i}}^n - \ov{\eta_{k,i}}^n).
  \]
  If $\xi_{k,i}, \eta_{k,i}$ are both real we have
  $|\xi_{k,i}^n - \eta_{k,i}^n| = \delta_{1,k,i}^{\, n} - \delta_{2,k,i}^{\, n}$, where
  $\delta_{1,k,i} = \max\{| \xi_{k,i} |, |\eta_{k,i}|\} $ and
  $\delta_{2,k,i}=\frac{\xi_{k,i} \, \cdot \, \eta_{k,i}}{\delta_{1,k,i}}$.
  Reorder the eigenvalues so that for $i=1,\dots,s$ at least one of $\xi_{k,i}, \eta_{k,i}$
  has nonzero imaginary part, the complex conjugates have subsequent indices
  and for $i=s+1,\dots,d_k$ both eigenvalues are real. Then
  \begin{equation}
      N(f_k^n, g_k^n)
      = \prod_{i=1}^{d_k} | \xi_{k,i}^n - \eta_{k,i}^n|
      = \sum \pm (\mu_1 \nu_1 \cdots \mu_{s/2} \nu_{s/2}
      \delta_{\epsilon(1),k,s+1} \cdots \delta_{\epsilon(T),k,d_k})^n,
  \end{equation}
  where $\mu_i$ is $\xi_{k,i}$ or $\eta_{k,i}$, $\nu_i$ is $\ov{\xi_{k,i}}$ or $\ov{\eta_{k,i}}$,
  $\epsilon(i)$ is $1$ or $2$.
  Substituting this for every $k=1,\dots,c$ to (\ref{eq:nielsen-romankov}) we get
  \begin{equation}\label{eq:nielsen-sum}
      N(f^n, g^n) = \sum_{j \in J} c_jw_j^{\, n},
  \end{equation}
  where $J$ is a finite index set, $c_j =\pm 1$ and $ w_j \in \mathbb{C} \setminus \{0\}$, $j\in J$.
  Consequently, the Nielsen coincidence zeta function
  \begin{align*}
    N_{f,g}(z) &= \exp \left( \sum_{i=1}^\infty \frac{N(f^n, g^n)}{n} z^n \right)
    = \exp \left( \sum_{j \in J} c_j \sum_{n=1}^\infty \frac{(w_jz)^n}{n} \right) \\
    &= \prod_{j \in J} (1 - w_jz)^{-c_j}
  \end{align*}
  is rational.

  To prove (iv) note that in (\ref{eq:nielsen-sum}) there is a dominant term $w_m$ for which
  $|w_m| = \prod_{k=1}^c \prod_{i=1}^{d_k} \max \{|\xi_{k,i}|, |\eta_{k,i}|\} > |w_j|$ for $j \neq m$.
  We get $\lim_{n \to \infty} \sqrt[n]{N(f^n, g^n)} = |w_m|$ which is \eqref{eq:nielsen-growth-formula}.

  To prove (v) observe that the assumption of tameness gives us equality of zeta functions
  $N_{f,g}(z) \equiv R_{f_\#,g_\#}(z)$.
  By (iii) these zeta functions are rational.
  $R_{f_\#,g_\#}(z)$ fulfills assumptions of Theorem \ref{thm:rational-congruences}
  and the congruences (\ref{eq:nielsen-gauss-congruences}) follow.
\end{proof}

\bigskip
\section{Asymptotic behavior of the sequences $\{R(\varphi^k, \psi^k)\}$
    and $\{N(f^k, g^k)\}$}\label{sec:asymptotic}

Assume that Reidemeister coincidence zeta function is rational.
Then the equation (\ref{eq:coincidence-sum}) $R(\varphi^k, \psi^k) = \sum_{i=1}^{r} \chi_i \lambda_i^k$
holds and from the proof of Lemma \ref{thm:rational-zeta-coefficients-sum} we know that $\lambda_i^{-1}$
are zeros and poles of the zeta function.
Define
\begin{equation}\label{eq:lambda-def}
    \lambda(\varphi, \psi) := \max_i \left\{ |\lambda_i| \right\}, \quad
    n(\varphi, \psi) := \#\{ \lambda_i \mid |\lambda_i| = \lambda(\varphi, \psi) \}.
\end{equation}
In theorem \ref{thm:main-result-nilpotent} we proved a formula for growth of Reidemeister coincidence numbers.
Providing $R_{\varphi, \psi}(z)$ is rational we can say more about the structure of the set of limit points
of the sequence $\left\{ R(\varphi^k, \psi^k) \ / \ \lambda(\varphi, \psi)^k \right\}_{k=1}^\infty$.
Babenko and Bogatyi stated the following theorem.

\begin{theorem}[\protect{\cite[Theorem 2.6]{BaBo}, see also \cite[Theorem 3.1.53]{JezMar}}]\label{thm:BaBo-trichotomy}
  If $f$ is a map of a compact ANR $X$ to itself and $r(f)$ is the reduced (or essential) spectral radius
  of the map $f_*$ induced on the homology group, then one of the following possibilities holds:
  \begin{enumerate}
      \item[(i)] $L(f^k) = 0, \ k=1,2,\dots$ .
      \item[(ii)] The set of limit points of the sequence $\left\{ L(f^k) \ / \ r(f)^k \right\}_{k=1}^\infty$
        contains an interval.
      \item[(iii)] The sequence $\left\{ L(f^k) \ / \ r(f)^k \right\}_{k=1}^\infty$ has the same limit points
        as the periodic sequence $\left\{ \sum_{\epsilon_i^m=1} A_i \epsilon_i^k \right\}_{k=1}^\infty$,
        where the $A_i$ are integers which do not vanish simultaneously.
  \end{enumerate}
\end{theorem}

Fel'shtyn and Lee proved an analogous result for Nielsen numbers of a map of an infra-solvmanifold of type $(R)$ \cite[Theorem 4.1]{FelsLee}.
We are going to prove similar result for Reidemeister coincidence numbers of endomorphisms
of a torsion-free nilpotent group of finite Pr\"ufer rank.
Denote the radius of convergence of $R_{\varphi, \psi}(z)$ by $\mathcal{R}$.

\begin{lemma}\label{thm:lambda-inverse-reidemeister}
  $\lambda(\varphi, \psi) = 1 / \mathcal{R}$.
\end{lemma}
\begin{proof}
  From the Cauchy-Hadamard formula there is
  \[
    \frac{1}{\mathcal{R}} = \limsup_{k \to \infty} \sqrt[k]{|R(\varphi^k, \psi^k)|}.
  \]
  Taking into account the equality (\ref{eq:coincidence-sum}) there is
  \begin{align*}
    \sqrt[k]{R(\varphi^k, \psi^k)} &= \sqrt[k]{\sum_{|\lambda_i| = \lambda(\varphi, \psi)} \chi_i \lambda_i^k
        + \sum_{|\lambda_i| < \lambda(\varphi, \psi)} \chi_i \lambda_i^k } \\
     &= \lambda(\varphi, \psi) \cdot \sqrt[k]{\sum_{|\lambda_i| = \lambda(\varphi, \psi)} \chi_i \cdot C_i^k
         + \sum_{|\lambda_i| < \lambda(\varphi, \psi)} \chi_i \left( \frac{\lambda_i}{\lambda(\varphi, \psi)} \right)^k },
  \end{align*}
  where $C_i \in \C$ are constants such that $|C_i| = 1$ for all $i$.
  The module of the sum under the root is bounded,
  therefore the module of the whole root goes to $1$ when $k \to \infty$.
  The lemma is proved.
\end{proof}

\begin{remark}
  Because directly from the definition there is $R(\varphi^k, \psi^k) \geq 1$,
  the case of Theorem \ref{thm:BaBo-trichotomy} (a) for Reidemeister coincidence numbers is impossible.
  Moreover based on the lemma above we have $\lambda(\varphi, \psi) \geq 1$.
\end{remark}

\begin{theorem}\label{thm:coincidence-reidemeister-asymptotic-sequence}
  Let $\varphi, \psi: G \to G$ be a tame pair of endomorphisms of a torsion-free nilpotent group $G$
  of finite Pr\"ufer rank and let the Reidemeister coincidence zeta function $R_{\varphi, \psi}(z)$ be rational.
  One of the two possibilities holds:
  \begin{enumerate}
    \item[(i)] The sequence $\{ R(\varphi^k, \psi^k) \ / \ \lambda(\varphi, \psi)^k \}_{k=1}^\infty$ has the same
      limit points set as a periodic sequence $\{ \sum_{j=1}^{n(\varphi, \psi)} \alpha_j \epsilon_j^k \}_{k=1}^\infty$
      where $\alpha_j \in \Z, \ \epsilon_j \in \C$ and $\epsilon_j^q = 1$ for some integer $q>0$.
    \item[(ii)] The set of limit points of the sequence $\{ R(\varphi^k, \psi^k) \ / \ \lambda(\varphi, \psi)^k \}_{k=1}^\infty$
      contains an interval.
  \end{enumerate}
\end{theorem}
\begin{proof}
  We factor every $\lambda_i$ for which $|\lambda_i| = \lambda(\varphi, \psi)$
  in \eqref{eq:coincidence-sum} as
  $$\lambda_i = \lambda(\varphi, \psi) \exp(2\pi\iota\theta_i),$$
  where $\iota^2 = -1$ is the imaginary unit.
  Then
  \[
    R(\varphi^k, \psi^k) = \lambda(\varphi, \psi)^k
      \left( \sum_{|\lambda_i| = \lambda(\varphi, \psi)} \chi_i \exp(2\pi\iota k \theta_i) \right)
      + \sum_{|\lambda_i| < \lambda(\varphi, \psi)} \chi_i \lambda_i^k
  \]
  and
  \[
    \left| \frac{R(\varphi^k, \psi^k)}{\lambda(\varphi, \psi)^k}
      - \sum_{|\lambda_i| = \lambda(\varphi, \psi)} \chi_i \exp(2\pi\iota k \theta_i) \right|
    \leq \sum_{|\lambda_i| < \lambda(\varphi, \psi)} |\chi_i|
      \cdot \left| \frac{\lambda_i}{\lambda(\varphi, \psi)} \right|^k.
  \]
  Right hand side goes to $0$ when $k \to \infty$ therefore the sequences
  \[
    \left\{ R(\varphi^k, \psi^k) \ / \ \lambda(\varphi, \psi)^k \right\}_{k=1}^\infty
    \mbox{\quad and \quad }
    \left\{ \sum_{|\lambda_i| = \lambda(\varphi, \psi)} \chi_i \exp(2\pi\iota k \theta_i) \right\}_{k=1}^\infty
  \]
  have the same asymptotic behavior.
  Now we claim that the structure of the limit points set of the latter sequence
  depends on rationality of all $\theta_i$.

  Assume that for all $i = 1, \dots, n(\varphi, \psi)$ we can write $\theta_i = p_i / q_i$ where $p_i, q_i \in \Z$.
  Put $q := \operatorname{lcm}(q_1, \dots, q_{n(\varphi, \psi)})$.
  Whenever $k$ is a multiplicity of $q$, all $\exp(2\pi\iota k \theta_i) = 1$.
  This is the case (i).

  Now assume that some of $\theta_i$ are irrational and denote the subset of their indices
  by $\mathcal{S} := \left\{ i_1, \dots, i_s \right\} \subset \{1, \dots, n(\varphi, \psi)\}$.
  We can express the assumption as linear independence
  \[
    \left( \sum_{i \in \mathcal{S}} \alpha_i \theta_i = \alpha, \quad \alpha_i, \alpha \in \Z \right) \qquad
    \Longrightarrow \qquad \forall_i \ \alpha_i = \alpha = 0.
  \]
  Then the Kronecker theorem (see \cite[Theorem VIII.6]{Cha}) states that the sequence \linebreak
  $\left\{ (k\theta_{i_1}, \dots, k\theta_{i_s}) \right\}_{k=1}^\infty$ is dense
  in the $s$-dimensional torus $\mathbb{T}^s$.

  Consider the continuous function
  \[
    h: \mathbb{T}^{n(\varphi, \psi)} \to [0, \infty), \qquad
    h(\xi_1, \dots, \xi_{n(\varphi, \psi)}) = \left| \sum_{i=1}^{n(\varphi, \psi)} \chi_i \exp(2\pi\iota \xi_i) \right|.
  \]
  We can consider the subtorus
  \[
    \mathbb{T}_{\mathcal{S}} := \left\{ (\xi_1, \dots, \xi_{n(\varphi, \psi)})
      \in \mathbb{T}^{n(\varphi, \psi)} \mid \xi_j = 0, \forall j \notin \mathcal{S} \right\}
  \]
  For $\mathbb{T}_{\mathcal{S}}$ is compact and the restriction
  $h_\mathcal{S} := \left. h \right|_{\mathbb{T}_\mathcal{S}}$
  is continuous, there exists the maximum $m_\mathcal{S} := \max_{\mathbb{T}_{\mathcal{S}}} h_\mathcal{S}(\xi)$.
  Observe that the closure of the image of the dense sequence contains the interval
  \[
    \overline{h_\mathcal{S}\big( \left\{ k\theta_{i_1}, \dots, k\theta_{i_s} \right\}_k \big)}
    \ \supset \ [0, m_\mathcal{S}].
  \]
  This is the case (ii).
\end{proof}

Let $f,g: N \to N$ be a tame pair of maps of a compact nilmanifold to itself
as discussed in Chapter \ref{sec:coincidence-nielsen}.
Theorem \ref{thm:nielsen-reidemeister} establishes the~equality
$R(f_\#, g_\#) = N(f,g)$ of Reidemeister coincidence numbers of the tame pair
of endomoprhism $f_\#, g_\#$ induced on the fundamental group of~$N$ and Nielsen coincidence numbers.
Using \eqref{eq:lambda-def}, if the Reidemeister coincidence zeta function $R_{f_\#,g_\#}(z)$
is rational we can define
\begin{align}\label{eq:lambda-nielsen}
  \lambda(f,g) := \lambda(f_\#, g_\#)
\end{align}
for the tame pair $f,g$.
Basing on Lemma \ref{thm:lambda-inverse-reidemeister}
we can interpret $\lambda(f,g)$ as the inverse of the radius of convergence
of~the~Nielsen coincidence zeta function $N_{f,g}(z)$.

Results of Theorem \ref{thm:main-result-nielsen-coincidence}
and Theorem \ref{thm:coincidence-reidemeister-asymptotic-sequence} let us formulate the following
\begin{theorem}
  Let $f,g: N \to N$ be a tame pair of maps of a compact nilmanifold to itself.
  Under assumptions of {\rm Theorem \ref{thm:main-result-nielsen-coincidence}
  (iii)} one of the two possibilities holds:
  \begin{enumerate}
    \item[(i)] The sequence $\{ N(f^k,g^k) \ / \ \lambda(f, g)^k \}_{k=1}^\infty$ has the same
      limit points set as a periodic sequence $\{ \sum_{j=1}^{n(\varphi, \psi)} \alpha_j \epsilon_j^k \}_{k=1}^\infty$
      where $\alpha_j \in \Z, \ \epsilon_j \in \C$ and $\epsilon_j^q = 1$ for some integer $q>0$.
    \item[(ii)] The set of limit points of the sequence $\{ N(f^k, g^k) \ / \ \lambda(f, g)^k \}_{k=1}^\infty$
      contains an interval.
  \end{enumerate}
\end{theorem}
\begin{proof}
  Because the pair $f,g$ is tame, we have the equality $N(f^k,g^k) = R(f_\#^k,g_\#^k)$ for all iterations $k \in \N$
  and equality of zeta functions $N_{f,g}(z) \equiv R_{f_\#,g_\#}(z)$ as well.
  Under assumptions of {\rm Theorem \ref{thm:main-result-nielsen-coincidence}
  (iii)} both zeta functions are rational.
  We can apply Theorem~\ref{thm:coincidence-reidemeister-asymptotic-sequence}
  to the sequence $\{ R(f_\#^k,g_\#^k) \ / \ \lambda(f_\#,g_\#)^k\}_{k=1}^\infty$
  which in our case is identical to the sequence $\{ N(f^k,g^k) \ / \ \lambda(f, g)^k \}_{k=1}^\infty$.
\end{proof}

\begin{remark}
  Consider a pair $f,g$ such that $N(f^k,g^k)=0$ for all iterations $k \in \N$.
  From {\rm Theorem \ref{thm:nielsen-reidemeister}} we can see that the pair is not tame
  and the Reidemeister coincidence zeta function $R_{f_\#,g_\#}(z)$ is not defined.
  On the other hand, the Nielsen coincidence zeta function $N_{f,g}(z) \equiv 1$ is constant, hence rational.
  The radius of convergence is infinite therefore it makes sense to set $\lambda(f,g)=0$
  according to our intepretation of \eqref{eq:lambda-nielsen} as the inverse of the radius of convergence.
  This situation is then an analogue of the third possibility for Lefschetz numbers
  in {\rm Theorem \ref{thm:BaBo-trichotomy} (i)} and in {\rm \cite[Theorem 4.1 (1)]{FelsLee}}
  for Nielsen numbers of maps of infra-solvmanifolds of type {\rm (R)}.
\end{remark}

\end{document}